\documentclass[10pt]{amsart}
\usepackage{amssymb,mathrsfs,skak}

\def\bC {\mathbf{C}}

\def\bN {\mathbf{N}}
\def\bQ {\mathbf{Q}}
\def\bR {\mathbf{R}}

\def\bZ {\mathbf{Z}}

\def\cF {\mathcal{F}}

\def\cN {\mathcal{N}}
\def\cO {\mathcal{O}}
\def\cP {\mathcal{P}}

\def\cS {\mathcal{S}}

\def\scrH{\mathscr{H}}
\def\scrL{\mathscr{L}}

\def\a {{\alpha}}
\def\b {{\beta}}

\def\de {{\delta}}
\def\eps {{\epsilon}}
\def\th {{\theta}}
\def\ka {{\kappa}}
\def\l {{\lambda}}

\def\om {{\omega}}
\def\Om {{\Omega}}

\def\d {{\partial}}
\def\grad {{\nabla}}
\def\Dlt {{\Delta}}

\def\rstr {{\big |}}
\def\indc {{\bf 1}}

\def\la {\langle}
\def\ra {\rangle}
\def \La {\bigg\langle}
\def \Ra {\bigg\rangle}

\newcommand{\Div}{\operatorname{div}}

\newcommand{\Sign}{\operatorname{sign}}

\newcommand{\Supp}{\operatorname{supp}}
\newcommand{\Det}{\operatorname{det}}

\newcommand{\Dist}{\operatorname{dist}}

\newcommand{\Id}{\operatorname{Id}}
\newcommand{\Lip}{\operatorname{Lip}}

\newcommand{\ba}{\begin{aligned}}
\newcommand{\ea}{\end{aligned}}

\newcommand{\be}{\begin{equation}}
\newcommand{\ee}{\end{equation}}

\newcommand{\lb}{\label}

\newtheorem{Thm}{Theorem}[section]
\newtheorem{Expl}[Thm]{Example}

\newtheorem{Cor}[Thm]{Corollary}
\newtheorem{Lem}[Thm]{Lemma}
\newtheorem{Def}[Thm]{Definition}

\begin{document}

\title[Monokinetic Measures with Rough Momentum Profiles]{Hamiltonian Evolution of Monokinetic Measures with Rough Momentum Profile}

\author[C. Bardos]{Claude Bardos}
\address[C.B.]{Universit\'e Paris-Diderot, Laboratoire J.-L. Lions, BP187, 4 place Jussieu, 75252 Paris Cedex 05 France}
\email{claude.bardos@gmail.com}

\author[F. Golse]{Fran\c cois Golse}
\address[F.G.]{Ecole Polytechnique, Centre de Math\'ematiques Laurent Schwartz (CMLS), 91128 Palaiseau Cedex, France}
\email{francois.golse@math.polytechnique.fr}

\author[P. Markowich]{Peter Markowich}
\address[P.M.]{King Abdullah University of Science and Technology (KAUST), MCSE Division, Thuwal 23955-6900, Saudi Arabia}
\email{Peter.Markowich@kaust.edu.sa}

\author[T. Paul]{Thierry Paul}
\address[T.P.]{CNRS and Ecole Polytechnique, Centre de Math\'ematiques Laurent Schwartz (CMLS), 91128 Palaiseau Cedex, France}
\email{thierry.paul@math.polytechnique.fr}

\begin{abstract}
Consider a monokinetic probability measure on the phase space $\bR^N_x\times\bR^N_\xi$, i.e. $\mu^{in}=\rho^{in}(x)\de(\xi-U^{in}(x))$ where $U^{in}$ is a 
vector field on $\bR^N$ and $\rho^{in}$ a probability density on $\bR^N$. Let $\Phi_t$ be a Hamiltonian flow on $\bR_N\times\bR^N$. In this paper, we study 
the structure of the transported measure $\mu(t):=\Phi_t\#\mu^{in}$ and of its integral in the $\xi$ variable denoted $\rho(t)$. In particular,  we give estimates 
on the number of folds in $\Phi_t(\hbox{ graph of $U^{in}$})$, on which $\mu_(t)$ is concentrated. We explain how our results can be applied to investigate the 
classical limit of the Schr\"odinger equation by using the formalism of Wigner measures. Our formalism includes initial momentum profiles $U^{in}$ with much 
lower regularity than required by the WKB method. Finally, we discuss a few examples showing that our results are sharp.
\end{abstract}



\maketitle


\section{Motivation}\lb{S-Motiv}


Consider the Cauchy problem for the Schr\"odinger equation in the case of a quantistic particle whose motion is driven by a potential $V\equiv V(x)\in\bR$:
\be\lb{ClassSchrod}
\left\{
\ba
{}&i\eps\d_t\psi_\eps=-\tfrac12\eps^2\Dlt_x\psi_\eps+V(x)\psi_\eps\,,\quad x\in\bR^N\,,\,\,t\in\bR\,,
\\
&\psi_\eps(0,x)=\psi^{in}(x)\,.
\ea
\right.
\ee
The unknown is the particle's wave function $\psi_\eps\equiv\psi_\eps(t,x)\in\bC$, while $\psi^{in}$ designates the initial datum. The dimensionless parameter 
$\eps>0$ is the ratio of the particle's de Broglie wave length (see \cite{Landau3} on p. 51) to some macroscopic observation length scale.

We choose the initial datum $\psi^{in}$ in the form of a WKB ansatz
\be\lb{ScalIn}
\psi^{in}(x):=a^{in}(x)e^{iS^{in}(x)/\eps}\,,
\ee
where $a^{in}$ and $S^{in}$ are real-valued, measurable functions on $\bR^N$ such that
$$
\int_{\bR^N}a^{in}(x)^2dx=1\,.
$$
The classical limit of quantum mechanics for a particle subject to the potential $V$ corresponds with the asymptotic behavior of the wave function $\psi_\eps$ 
as $\eps\to 0^+$ (see chapter VII in \cite{Landau3}).

Assume that $V\in C^\infty(\bR^N)$ is such that $-\tfrac12\eps^2\Dlt+V$ has a self-adjoint extension to $L^2(\bR^N)$ which is bounded from below, while
\be\lb{HypV}
V(x)=o(|x|)\hbox{ and }\d^\a V(x)=O(1)\hbox{ for each }\a\in\bN^N\hbox{ as }|x|\to\infty\,.
\ee
Let $\Phi_t $ be the Hamiltonian flow 
$$
(x,\xi)\mapsto\Phi_t(x,\xi)=(X_t(x,\xi),\Xi_t(x,\xi))
$$
defined by the function $H(x,\xi):=\tfrac12|\xi|^2+V(x)$ on the phase space $\bR^N_x\times\bR^N_\xi$ equipped with the canonical symplectic $2$-form 
$d\xi_1\wedge dx_1+\ldots+d\xi_N\wedge dx_N$. In other words, $t\mapsto(X_t(x,\xi),\Xi_t(x,\xi))$ is the integral curve of the Hamiltonian vector field
$(X,\Xi)\mapsto(\Xi,-\grad V(X))$ passing through $(x,\xi)$ at time $t=0$.

Assuming that $S^{in}\in C^2(\bR)$, consider the $C^1$ map
$$
F_t:\,\bR^N\ni y\mapsto X_t(y,\grad S^{in}(y))\in\bR^N
$$
for each $t\in\bR$,  and set $J_t(y):=|\Det(DF_t(y))|$. 

Denote by $C_t$ the set of critical values of $F_t$, which is Lebesgue negligible by Sard's theorem, and set $C:=\{(t,x)\,|\,x\in C_t\,,\,\,t\in\bR\}$.

Under assumption (\ref{HypV}) and provided that
\be\lb{HypS}
\grad S^{in}(y)=o(|y|)\hbox{ as }|y|\to\infty\,,
\ee
for each $t\in\bR$ and each $x\in\bR^N\setminus C_t$, the equation $F_t(y)=x$ has finitely many solutions denoted $y_j(t,x)$ for $j=1,\ldots,\cN(t,x)$.

The following result is based on global asymptotic methods developed by Maslov \cite{Maslov} and stated without proof as Theorem 5.1 in \cite{ArnoldMaslov}.

\begin{Thm}\lb{T-WKB}
Let $a^{in}\in C^m_c(\bR^N)$ and $S^{in}\in C^{m+1}(\bR^N)$ with $m>6N+4$. For all $\eps>0$ and all $(t,x)\in\bR_+\times\bR^N\setminus C$, set
\be\lb{WKBSol}
\Psi_\eps(t,x)\!:=\!\sum_{j=1}^{\cN(t,x)}\!\frac{a^{in}(y_j(t,x))}{J_t(y_j(t,x))^{1/2}}e^{iS_j(t,x)/\eps}e^{i\pi\nu_j(t,x)/2}\,,
\ee
where, for each $j=1,\ldots,\cN(t,x)$,
$$
S_j(t,x)\!:=\!S^{in}(y_j(t,x))\!+\!\!\int_0^t\left(\tfrac12|\Xi_s(y_j(t,x),\grad S^{in}(y_j(t,x)))|^2\!-\!V(F_s(y_j(t,x)))\right)ds
$$
and $\nu_j(t,x)\in\bZ$ is the Morse index of the path $[0,t]\ni s\mapsto F_s(y_j(t,x))\in\bR^N$. 

Then the solution $\psi_\eps$ of the Cauchy problem (\ref{ClassSchrod}) satisfies
\be\lb{WKBLimV}
\psi_\eps(t,x)=\Psi_\eps(t,x)+R_\eps^1(t,x)+R_\eps^2(t,x)
\ee
for all $T>0$, where
$$
\sup_{0\le t\le T}\|R_\eps^1\|_{L^2(B(0,R))}+\sup_{(t,x)\in K}|R_\eps^2(t,x)|=O(\eps)
$$
for each $R>0$ and each compact $K\subset(\bR\times\bR^N)\setminus C$.
\end{Thm}

Notice that each phase function $S_j$ is a solution of the Hamilton-Jacobi eikonal equation
$$
\d_tS_j(t,x)+\tfrac12|\grad_xS(t,x)|^2+V(x)=0
$$

Theorem \ref{T-WKB} shows that the single WKB ansatz (\ref{ScalIn}) evolves under the dynamics of the scaled Schr\"odinger equation (\ref{ClassSchrod})
into a wave function $\psi_\eps$ that is asymptotically close to a locally finite sum of WKB ansatz for $x\notin C_t$ as $\eps\to 0^+$.

\smallskip
There is a different approach to the classical limit of quantum mechanics, that is based on the notion of Wigner transform \cite{Wigner}. The Wigner transform
at scale $\eps>0$ of a complex valued function $\phi\in L^2(\bR^N)$ is
$$
W_\eps[\phi](x,\xi):=\tfrac1{(2\pi)^N}\int_{\bR^N}\phi(x+\tfrac12\eps y)\overline{\phi(x-\tfrac12\eps y)}e^{-i\xi\cdot y}dy
$$
Theorem IV.1 and Example III.5 in \cite{LionsPaul} are summarized in the following statement --- see also \cite{PG}.

\begin{Thm}\lb{T-Wigner} Let $a^{in}\in L^2(\bR^N)$ and $S^{in}\in W^{1,1}_{loc}(\bR^N)$ with $\|a^{in}\|_{L^2(\bR^N)}=1$. Then
$$
W_\eps\left[a^{in}e^{iS^{in}/\eps}\right](x,\cdot)\to|a^{in}(x)|^2\de_{\grad S^{in}(x)}\hbox{ in }\cS'(\bR^N_x\times\bR^N_\xi)
$$
as $\eps\to 0$. 

Assume that $V$ is such that $-\tfrac12\eps^2\Dlt+V$ has a self-adjoint extension to $L^2(\bR^N)$ which is bounded from below, and satisfies 
$$
V\in C^{1,1}(\bR^N)\quad\hbox{ and }V(x)\ge-C(1+|x|^2)
$$
for some constant $C>0$. Then the solution $\psi_\eps$ of the Cauchy problem (\ref{ClassSchrod}) with initial datum (\ref{ScalIn}) satisfies
$$
W_\eps[\psi_\eps(t,\cdot)]\to\mu(t)\quad\hbox{ in }\cS'(\bR^N\times\bR^N)\hbox{ uniformly in }t\in[0,T]
$$
for each $T>0$ as $\eps\to 0$, where $\mu\in C_b(\bR;w-\cP(\bR^N\times\bR^N))$ is the solution of the Cauchy problem for the Liouville equation of classical
mechanics
$$
\left\{
\ba{}
&\d_t\mu+\xi\cdot\grad_x\mu-\Div_\xi(\mu\grad_xV(x))=0\,,\quad x,\xi\in\bR^N\,,\,\,t\in\bR\,,
\\
&\mu(0,x,\cdot)=|a^{in}(x)|^2\de_{\grad S^{in}(x)}\,.
\ea
\right.
$$
\end{Thm}

The initial datum in the Cauchy problem for the Liouville equation 
$$
\mu^{in}(x,\cdot):=a^{in}(x)^2\de_{\grad S^{in}(x)}
$$ 
is an example of a ``monokinetic measure'' on the phase space $\bR^N_x\times\bR^N_\xi$. According to Theorem \ref{T-Wigner}, the Wigner transform of the 
solution of the Cauchy problem for the scaled Schr\"odinger equation (\ref{ClassSchrod}) with WKB initial data (\ref{ScalIn}) converges as $\eps\to 0^+$ to 
$\mu(t)=\Phi_t\#\mu^{in}$, the initial monokinetic measure transported along trajectories of classical mechanics.

Comparing the assumptions in Theorems \ref{T-WKB} and \ref{T-Wigner} shows that the approach of the classical limit of quantum mechanics based on the
Wigner transform requires much less regularity on both the potential $V$ and the initial amplitude $a^{in}$ and phase $S^{in}$. On the other hand, Theorem
\ref{T-WKB} provides detailed information on the structure of the evolved wave function $\psi_\eps(t)$ for all $t\in\bR$. This suggests the following question.

\smallskip
\noindent
\textbf{Problem A.} To find the structure of $\mu(t)=\Phi_t\#\mu^{in}$, where $\mu^{in}$ is a monokinetic measure and $\Phi_t$ a Hamiltonian flow on the phase 
space on $\bR^N_x\times\bR^N_\xi$.  For instance, is $\mu(t) $ a locally finite sum of monokinetic measures?

\smallskip
This last question is suggested by the following observation. According to Proposition 1.5 in \cite{GMMP} and in view of Example III.5 in \cite{LionsPaul}, for all 
$a_1,\ldots,a_n\in L^2(\bR^N)$ and all $S_1,\ldots,S_n\in W^{1,1}_{loc}(\bR^N)$ such that $\grad S_1(x),\ldots,\grad_xS_n(x)$ are linearly independent in 
$\bR^N$ for a.e. $x\in\bR^N$, one has
\be\lb{SumWigner}
W_\eps\left[\sum_{k=1}^na_ke^{iS_k/\eps}\right](x,\cdot)\to\sum_{k=1}^na_k(x)^2\de_{\grad S_j(x)}\hbox{ in }\cS'(\bR^N\times\bR^N)\hbox{  as }\eps\to 0\,.
\ee
However, one should not seek to apply Theorem \ref{T-WKB} and the observation above to answer Problem A, since the locally finite sum $\Psi_\eps(t,\cdot)$ 
of WKB ansatz is known to approximate the solution $\psi_\eps(t,\cdot)$ of the Cauchy problem (\ref{ClassSchrod}) away from the set $C_t$ only and not 
globally in the Euclidean space $\bR^N$. 

A preliminary, more formal study of the structure of Wigner measures evolving from monokinetic measures associated to initial WKB wave functions can be 
found in \cite{SparMM}. The main purpose of the present study is to go further in this direction, and especially to answer the problems posed in this introduction 
under less stringent regularity assumptions than those used in Theorem \ref{T-WKB}. 

\smallskip
A related question, bearing on the role of the set $C_t$ in Theorem \ref{T-WKB}, is the following problem.

\smallskip
\noindent
\textbf{Problem B.} To find the structure of the first marginal $\rho(t)$ of the probability measure $\mu(t)$ on $\bR^N_x\times\bR^N_\xi$. 

\smallskip
Define $\Lambda^{in}:=\{(x,\grad S^{in}(x))\,|\,x\in\bR^N\}$ and $\Lambda_t:=\Phi_t(\Lambda^{in})$. Under the same assumptions as in Theorem \ref{T-WKB}, 
by definition of the functions $\cN(t,x)$ and $y_j(t,x)$ for $j=1,\ldots,\cN(t,x)$, one has
\be\lb{Lambdat=}
\Lambda_t\cap((\bR^N\setminus C_t)\times\bR^N)=\bigcup_{j\ge 1}\{(y_j(t,x),\grad S^{in}(y_j(t,x)))\,|\,x\hbox{ is s.t. }\cN(t,x)\ge j\}\,.
\ee
In other words, the number of WKB terms in the asymptotic formula $\Psi_\eps$ for the solution $\psi_\eps$ of the Cauchy problem (\ref{ClassSchrod}) in
quantum mechanics with initial datum (\ref{ScalIn}) has a geometric interpretation in terms of the classical Hamiltonian $\Phi_t$. The previous equality shows
that the function $\cN$ measures the number of folds in $\Lambda_t$.  Notice that $\Lambda_t$ can be defined for all $t\in\bR$ under assumptions much 
weaker than those in Theorem \ref{T-WKB}. (For instance, the Hamiltonian flow $\Phi_t$ is a continuous map on $\bR^N_x\times\bR^N_\xi$ defined for all 
$t\in\bR$ under the assumption on $V$ in Theorem \ref{T-Wigner}, while $\Lambda^{in}$ is the graph of a continuous map for $S^{in}\in C^1(\bR^N)$.) This 
observation suggests the following question.

\smallskip
\noindent
\textbf{Problem C.} To estimate the number of folds in $\Lambda_t$ for all $t\in\bR$.

\smallskip
Notice that these problems bear exclusively on the propagation of monokinetic measures under the dynamics defined by the Liouville equation of classical 
mechanics. Problems A-C can be formulated without any reference to the classical limit of quantum mechanics --- which nevertheless remains the main
motivation for studying these problems here.

As noticed above, the regularity assumptions on $a^{in},S^{in}$ and $V$ used in Theorem \ref{T-Wigner} are much weaker than those in the WKB method 
summarized in Theorem \ref{T-WKB}. This suggests the idea of using the answers to problems A-C to study the classical limit of quantum mechanics in cases
where the assumptions of Theorem \ref{T-WKB}Êare not satisfied.

\smallskip
\noindent
\textbf{Problem D.} To study the propagation under the scaled Schr\"odinger equation (\ref{ClassSchrod}) of a WKB ansatz (\ref{ScalIn}) in the classical limit
$\eps\to 0^+$ under regularity assumptions on $a^{in},S^{in}$ and $V$ weaker than  the level of regularity assumed in Theorem \ref{T-WKB}. 

\smallskip
The paper is organized as follows: our main results are stated in sections \ref{S-Main1} and \ref{S-Expl}. Section \ref{S-Main1} contains our answers to
Problems A-D stated above. Section \ref{S-Expl} presents examples and counterexamples showing that the theorems in section \ref{S-Main1} are sharp.
From the quantum mechanical point of view, these examples and counterexamples correspond to wave functions whose Wigner measures are carried 
by highly singular subsets of the phase space. The proofs of all the results in sections \ref{S-Main1} are given in sections \ref{S-N}-\ref{S-Psi}, while the 
proofs of the statements in section \ref{S-Expl} are deferred to section \ref{S-ProofEx}. 


\section{Main results I: Answering Problems A-D}\lb{S-Main1}


\subsection{Assumptions on the dynamics and on the initial data}\lb{SS-DynInit}

First we specify the assumptions on the classical Hamiltonian dynamics used throughout the present paper. 

Although the motivation of our study is the classical limit of the scaled Schr\"odin\-ger operator $-\tfrac12\eps^2\Dlt+V$, it would have been equally legitimate to 
raise the same questions in the case of the operator $H(x,-i\eps\d_x)$ obtained by the Weyl quantization rule from a smooth Hamiltonian $H\equiv H(x,\xi)$ 
defined on $\bR^N\times\bR^N$. A first obvious condition to be imposed on the function $H$ in this case is that the resulting operator $H(x,-i\eps\d_x)$ admits 
a self-adjoint extension to $L^2(\bR^N)$ for each $\eps>0$. Therefore we do not restrict our attention to the only case where $H(x,\xi)=\tfrac12|\xi|^2+V(x)$.

Throughout the present paper, we assume that $H\equiv H(x,\xi)\in\bR$ is a $C^2$ function on $\bR^N\times\bR^N$ satisfying the following conditions: there 
exists $\ka>0$, and a function $h\in C(\bR;\bR_+)$ sublinear at infinity, i.e.
$$
\frac{h(r)}{r}\to0\quad\hbox{ as }r\to+\infty\,,
$$
such that
\be\lb{HypH}
\ba
{}&|\grad_\xi H(x,\xi)|\le\ka(1+|\xi|)
\\
&|\grad_xH(x,\xi)|\le h(|x|)+\ka|\xi|
\\
&|\grad^2H(x,\xi)|\le\ka
\ea
\ee
for all $(x,\xi)\in\bR^N\times\bR^N$.

By the Cauchy-Lipschitz theorem, the Hamiltonian $H$ generates a unique, global Hamiltonian flow denoted $\Phi_t$ on $\bR^N_x\times\bR^N_\xi$ for the 
canonical symplectic $2$-form $d\xi_1\wedge dx_1+\ldots+d\xi_N\wedge dx_N$. In other words, for each $(x,\xi)\in\bR^N\times\bR^N$, the integral curve
of the Hamiltonian vector field $(D_xH,-D_\xi H)$ passing through $(x,\xi)$ for $t=0$ is $t\mapsto\Phi_t(x,\xi)$. We systematically use the following notation for 
the flow $\Phi_t$:
\be\lb{DefPhi}
\Phi_t(x,\xi)=(X_t(x,\xi),\Xi_t(x,\xi))\,,\quad x,\xi\in\bR^N\,,\,\,\quad t\in\bR\,,
\ee

\smallskip
Next we specify the assumptions on the initial data. 

\begin{Def}\lb{D-Monokin}
The monokinetic measure on $\bR^N_x\times\bR^N_\xi$ with momentum profile $U$ and density $\rho$ is the positive Borel measure $\mu$ whose
 disintegration with respect to the Lebesgue measure on $\bR^N_x$ and the canonical projection $(x,\xi)\mapsto x$ is
$$
\mu(x,\cdot):=\rho(x)\de_{U(x)}\,.
$$
Here $U$ is a continuous vector field on $\bR^N$ and $\rho$ a nonnegative element of $L^1(\bR^N)$. In other words, for each test function 
$\phi\in C_c(\bR^N_x\times\bR^N_\xi)$, one has
$$
\iint_{\bR^N\times\bR^N}\phi(x,\xi)\mu(dxd\xi)=\int_{\bR^N}\phi(x,U^{in}(x))\rho(x)dx\,.
$$
\end{Def}

Notice that, at variance with the case considered in section \ref{S-Motiv}, we treat the case of monokinetic measures where the $\xi$-profile $U^{in}$ is not 
necessarily a gradient field.

\smallskip
Let $U^{in}\in C(\bR^N;\bR^N)$ satisfy the following sublinearity condition at infinity
\be\lb{Sublin}
\frac{|U^{in}(y)|}{|y|}\to 0\quad\hbox{ as }|y|\to 0\,.
\ee
Unless otherwise specified, we assume that its gradient (in the sense of distributions) $DU^{in}$ satisfies the condition
\be\lb{CondLN1}
\d_lU^{in}_k\rstr_{\Om}\in L^{N,1}(\Om)\hbox{ for each bounded open }\Om\subset\bR^N\,,
\ee
for all $k,l=1,\ldots,N$. We recall that a measurable function $f:\,\Om\to\bR$ belongs to the Lorentz space $L^{N,1}(\Om)$ if
$$
\int_0^\infty\left(\scrL^N(\{x\in\Om\,|\,|f(x)|\ge\l\})\right)^{1/N}d\l<\infty\,.
$$
By Theorem B in \cite{KKM}, the vector field $U^{in}$ is differentiable a.e. on $\bR^N$. We henceforth denote by $E$ the $\scrL^N$-negligible set defined as
\be\lb{DefE}
E:=\{y\in\bR^N\,|\,U^{in}\hbox{ is not differentiable at }y\}\,.
\ee

Along with the vector field $U^{in}$, we consider the map
\be\lb{DefF}
F_t:\,\bR^N\ni y\mapsto F_t(y):=X_t(y,U^{in}(y))\in\bR^N\,.
\ee
By the chain rule, for each $t\in\bR$, the map $F_t$ is differentiable on $\bR^N\setminus E$. We henceforth use the following elements of notation
\be\lb{DefJ}
J_t(y):=|\Det(DF_t(y))|\hbox{ for all }y\in\bR^N\setminus E\hbox{ and }t\in\bR\,,
\ee
and
\be\lb{DefZP}
\ba
P_t:=\{y\in\bR^N\setminus E\hbox{ s.t. }J_t(y)>0\}\,,
\\
Z_t:=\{y\in\bR^N\setminus E\hbox{ s.t. }J_t(y)=0\}\,.
\ea
\ee
We generalize as follows the definition of the set $C_t$ considered in section \ref{S-Motiv}:
\be\lb{DefCt}
C_t:=\{x\in\bR^N\hbox{ s.t. }F_t^{-1}(\{x\})\cap(Z_t\cup E)\not=\varnothing\}\,,
\ee
and define
\be\lb{DefC}
C:=\{(t,x)\in\bR\times\bR^N\hbox{ s.t. }x\in C_t\}\,.
\ee
The sets $C$ and $C_t$ are referred to as ``the caustic'' and ``the caustic fiber'' respectively.

Finally, we designate by $\rho^{in}$ a probability density on $\bR^N$, and we consider the monokinetic measure $\mu^{in}$ defined by 
\be\lb{DefMuin}
\mu^{in}(x,\cdot)=\rho^{in}(x)\de_{U^{in}(x)}\,.
\ee
Our purpose is to study the structure of 
\be\lb{DefMut}
\mu(t):=\Phi_t\#\mu^{in}\quad\hbox{ and }\rho(t):=\Pi\#\mu(t)\,,
\ee
where $\Pi$ is the canonical projection $\bR^N\times\bR^N\ni(x,\xi)\mapsto x\in\bR^N$. 

A few remarks on the definition (\ref{DefCt}) are in order. 

First, if $U^{in}$ satisfies (\ref{CondLN1}) but is not of class $C^1$, the equation $F_t(y)=x$ may have infinitely many solutions although $x$ is not a singular value 
of $F_t$, as shown by the example below.

\begin{Expl}
Set $N=1$ and $H(x,\xi)=\tfrac12\xi^2$, so that $F_t(y)=y+tU^{in}(y)$. Take 
$$
U^{in}(z):=z\sin(\ln|z|)\hbox{ if }z\not=0\,,\quad\hbox{while }U^{in}(0)=0\,.
$$
so that $U^{in}\in\Lip(\bR)\setminus C^1(\bR)$ and the nondifferentiability set $E=\{0\}$. Then for each $t$ such that $|t|>1$, the set $F_t^{-1}(\{0\})\cap[-L,L]$ 
is infinite for each $L>0$. Assume for instance that $t<-1$ while $L=\pi$; then 
$$
F_t^{-1}(\{0\})\cap[-\pi,\pi]=\{0\}\cup\{\pm y_n(t)\,|\,n\ge 0\}\cup\{\pm z_n(t)\,|\,n\ge 0\}\,,
$$
where
$$
y_n(t):=\exp(\arcsin(-1/t)-2\pi n)\hbox{ and }z_n(t):=\exp(\pi-\arcsin(-1/t)-2\pi n)\,,
$$
for $n\in\bN$. On the other hand $F_t(y)=1+t\sin\ln|y|+t\cos\ln|y|$ so that 
$$
|F'_t(y_n(t))|=|F'_t(z_n(t))|=\sqrt{t^2-1}\not=0\,.
$$
Hence $0$ is not a critical value of the restriction of $F_t$ to $[-\tfrac\pi2,\tfrac\pi2]$.
\end{Expl}

In view of the observation following (\ref{HypS}), this example suggests that the definition of the caustic fiber for non $C^1$ momentum profiles $U^{in}$ should 
indeed involve the nondifferentiability set $E$, as in (\ref{DefCt}).

Yet there is a definite arbitrariness in the definition of $C_t$. For instance, the nondifferentiability set $E$ could be replaced by any other Lebesgue negligible 
set $E'$ in the definition of $C_t$. Since the initial density $\rho^{in}\in L^1(\bR^N)$, the initial monokinetic measure $\mu^{in}$ is obviously independent 
of $E'$. By uniqueness of the solution of the Cauchy problem for the Liouville equation, the propagated measure $\mu(t)$ and its first marginal $\rho(t)$  are 
also independent of $E'$. Consistently with this observation, we have chosen the minimal regularity requirements on the initial momentum profile $U^{in}$ 
leading to a precise description of the structure of $\rho(t)$ on $\bR^N\setminus C_t$ and for which $C_t$ is Lebesgue negligible independently of the choice
of the Lebesgue negligible set $E'$ used in its definition. We refer to Example \ref{E-CantorFunc} and the discussion thereafter for a further analysis of this
issue. 

Along with the probability measures $\mu^{in}$ and $\mu(t)$, we consider the sets
\be\lb{DefLinLt}
\Lambda^{in}:=\{(x,U^{in}(x))\,|\,x\in\bR^N\}\,,\quad\hbox{ and }\Lambda_t:=\Phi_t(\Lambda^{in})\,,\quad t\in\bR\,.
\ee

In the context of the classical limit of quantum mechanics as in Theorem \ref{T-WKB}-\ref{T-Wigner}, the initial momentum profile $U^{in}=\grad S^{in}$ is a 
gradient field with $S^{in}$ of class $C^2$ at least, so that $\Lambda_t$ is a Lagrangian submanifold of the phase space for all $t\in\bR$. However, the 
fact that $\Lambda_t$ is a Lagrangian does not play any particular role in our analysis, as all our results hold without assuming that $U^{in}$ is a gradient 
field.

\subsection{Answering problem C: the equation $F_t(y)=x$ and the set $\Lambda_t$}

The definition of the asymptotic solution in Theorem \ref{T-WKB} shows the importance of the equation $F_t(y)=x$ for the unknown $y$, and especially of 
the number of solutions of this equation whenever this number is finite. For each $t\in\bR$, each $x\in\bR^N$ and each $R>0$, set
\be\lb{DefNNR}
\cN(t,x):=\#F_t^{-1}(\{x\})\quad\hbox{ and }\cN_R(t,x):=\#F_t^{-1}(\{x\})\cap\overline{B(0,R)}
\ee
whenever these sets are finite, and $\cN(t,x)=+\infty$ or $\cN_R(t,x)=+\infty$ otherwise.

\smallskip
Our first main result in the present paper is the following theorem, which solves Problem C and provides additional information on the equation $F_t(y)=x$.
The key new pieces of information are the estimates in statements (d) and (f) below.

\begin{Thm}\lb{T-NURough}
Assume that $U^{in}$ is a continuous vector field on $\bR^N$ satisfying (\ref{Sublin}).

\smallskip
\noindent
a) For each $t\in\bR$, the map $F_t$ is proper and onto, and satisfies the condition
$$
\sup_{|t|\le T}|F_t(y)-y|=o(|y|)\quad\hbox{ as }|y|\to\infty\,.
$$

\smallskip
Assume moreover that $U^{in}$ satisfies (\ref{CondLN1}). 

\smallskip
\noindent
b) For each $t\in\bR$, the function $J_t$ belongs to $L^1_{loc}(\bR^N)$.

\noindent
c) For each $t\in\bR$, the set $F_t^{-1}(\{x\})$ is finite for a.e. $x\in\bR^N$.

\noindent
d) For each $t\in\bR$, each $R>0$ and each $n\in\bN$, one has
$$
\scrL^N\left(\{x\in\bR^N\hbox{ s.t. }\cN_R(t,x)\ge n\}\right)\le\frac1ne^{N\ka|t|}\|1+|DU^{in}|\,\|^N_{L^N(B(0,R))}\,.
$$
e) For each $t\in\bR$, the set $C_t$ defined in (\ref{DefCt}) satisfies $\scrL^N(C_t)=0$.

\noindent
f) For each $T>0$, 
$$
\scrH^1(\{(t,y)\!\in\![-T,T]\!\times\!\bR^N\hbox{ s.t. }F_t(y)\!=\!x\})\!<\!+\infty\quad\hbox{ for a.e. }x\in\bR^N\,.
$$
\end{Thm}

\smallskip
The estimate in statement (d) bears on the localized variant of the counting function $\cN_R$ instead of $\cN$. With the additional information in statement (a),
one can show that $\cN(t,x)=\cN_R(t,x)$ for $R$ large enough, by the following observation.

\begin{Cor}\lb{C-NURough}
Assume that $U^{in}$ is a continuous vector field on $\bR^N$ satisfying (\ref{Sublin}) and (\ref{CondLN1}), and define
$$
M_T(R):=\sup_{|y|\ge R}\sup_{|t|\le T}\frac{|F_t(y)-y|}{|y|}\quad\hbox{ for all }T,R>0\,.
$$
Let $T>0$ and $R_T^*>0$ be s.t. $M_T(R_T^*)<\tfrac12$. 

\smallskip
\noindent
a) For each $t\in[-T,T]$ and each $R\ge R_T^*$, one has
$$
\cN(t,x)=\cN_R(t,x)\hbox{ for each }x\in\overline{B(0,\tfrac12R)}\,.
$$
b) For each $t\in[-T,T]$ and each $R\ge R_T^*$, one has 
$$
\scrL^N\left(\{x\in\overline{B(0,\tfrac12R)}\hbox{ s.t. }\cN(t,x)\ge n\}\right)\le\frac1ne^{N\ka|t|}\|1+|DU^{in}|\,\|^N_{L^N(B(0,R))}\,.
$$
\end{Cor}

\smallskip
If the regularity condition (\ref{CondLN1}) is replaced by the assumption that $U^{in}$ is of class $C^1$, one obtains additional information on the equation 
$F_t(y)=x$, especially on the set of $x$'s for which this equation has finitely many solutions, on the number of its solutions, and on the dependence of these 
solutions in $t$ and $x$. This additional information has been gathered in the next theorem for the sake of completeness. While statements (a-d) are more 
or less classical consequences of the implicit function theorem, we believe that statement (e) is new.

\begin{Thm}\lb{T-NUC1}
Assume that $U^{in}$ is a $C^1$ vector field on $\bR^N$ satisfying (\ref{Sublin}) and the condition
$$
|DU^{in}(y)|=O(|y|)\quad\hbox{ as }|y|\to\infty\,.
$$
a) The set $C$ is closed in $\bR\times\bR^N$.

\noindent
b) The set $F_t^{-1}(\{x\})$ is finite for all $(t,x)\in\bR\times\bR^N\setminus C$, and the counting function $\cN$ is constant in each connected component of 
$\bR\times\bR^N\setminus C$.

\noindent
c) Let $O_n$ be a connected component of $\bR\times\bR^N\setminus C$ s.t. $\cN=n$ on $O_n$. For each $j=1,\ldots,n$, there exists $y_j\in C^1(O_n,\bR^N)$ 
s.t.
$$
F_t^{-1}(\{x\})=\{y_1(t,x),\ldots,y_n(t,x)\}\quad\hbox{ for all }(t,x)\in O_n\,.
$$
d) There exists $a<0<b$ such that $C\cap((a,b)\times\bR^N)=\varnothing$ and $\cN=1$ on $(a,b)\times\bR^N$.

\noindent
e) For each $(t,x)\in\bR\times\bR^N\setminus C$, the positive integer $\cN(t,x)$ is odd.
\end{Thm}

\smallskip
The relation of the equation $F_t(y)=x$ for the unknown $y$ to $\Lambda_t$ is obvious: 
$$
(x,\xi)\in\Lambda_t\Leftrightarrow\hbox{ there exists }y\in\bR^N\hbox{ s.t. }(x,\xi)=\Phi_t(y,U^{in}(y))\,.
$$
Equivalently
\be\lb{Lt=}
\Lambda_t=\{(x,\xi)\in\bR^N\times\bR^N\hbox{ s.t. }\xi=\Xi_t(y,U^{in}(y))\hbox{ for some }y\in F_t^{-1}(\{x\})\}\,.
\ee

Under the assumptions in Theorem \ref{T-NUC1} (c), one has a fairly precise description of $\Lambda_t$ in terms of the counting function $\cN$ and of the
solutions $y_j\equiv y_j(t,x)$ of the equation $F_t(y)=x$. Indeed, defining
$$
\Lambda:=\{(t,x,\xi)\,|\,(x,\xi)\in \Lambda_t\,,\,\,t\in\bR\}\,,
$$
one has
$$
\Lambda_t\cap(O_n\times\bR^N)=\bigcup_{j=1}^n\{(x,\Xi_t(y_j(t,x),U^{in}(y_j(t,x)))\,|\,(t,x)\in O_n\}\,.
$$
In the more general case where $U^{in}$ satisfies only the weaker assumptions in Theorem \ref{T-NURough}, the equality (\ref{Lt=}) and statement (c) imply
that, for a.e. $x\in\bR^N$, 
$$
(\Lambda_t)_x:=\Lambda_t\cap(\{x\}\times\bR^N)\hbox{ is finite and }\#(\Lambda_t)_x\le\cN(t,x)\,.
$$
Therefore, in both cases, the counting function $\cN$ measures the number of folds in the set $\Lambda_t$. The counting function $\cN$ itself --- or more 
precisely the distribution of its values --- is estimated by statement (d) in Theorem \ref{T-NURough} and by statement (b) in Corollary \ref{C-NURough}.

A last remark is in order: by Theorem \ref{T-NUC1} b, the equation $F_t(y)=x$ has finitely many solutions if $x\notin C_t$. The converse is not true, in other
words, it may happen that the equation $F_t(y)=x$ has finitely many solutions for $x\in C_t$, as shown by the following elementary example.

\begin{Expl}
Set $N=1$ and $H(x,\xi)=\tfrac12\xi^2$ so that $F_t(y)=y+tU^{in}(y)$. Assume that $U^{in}$ satisfies (\ref{Sublin}) and is real analytic on $\bR$. Then, for each 
$t,x\in\bR$, the set $F_t^{-1}(\{x\})$ is finite.
\end{Expl}

Indeed, the map $F_t$ is proper because $U^{in}$ satisfies (\ref{Sublin}), so that the set $F_t^{-1}(\{x\})$ is compact for all $x\in\bR$. Since $y\mapsto F_t(y)-x$ 
is real analytic on $\bR$ and not identically $0$, the set $F_t^{-1}(\{x\})$ of its zeroes consists of isolated points. Then, the Bolzano-Weierstrass theorem implies 
that $F_t^{-1}(\{x\})$ is finite.

\subsection{Answering problems A-B: structure of $\mu(t)$ and $\rho(t)$}

Our main results on the structure of the probability measures $\mu(t)$ and $\rho(t)$ are summarized in the following theorem.
  
\begin{Thm}\lb{T-Mu(t)=}
Let $U^{in}$ be a continuous vector field on $\bR^N$ satisfying (\ref{Sublin}) and (\ref{CondLN1}). Let $t\in\bR$, and let $\rho^{in}$ be a probability distribution
on $\bR^N$. Let $\mu^{in}$ be the monokinetic measure (\ref{DefMuin}), let $\mu(t)$ and $\rho(t)$ be the Borel probability measures in (\ref{DefMut}).

\noindent
a) The following three properties are equivalent:
$$
\rho(t)(C_t)=0\Leftrightarrow\rho(t)(\bR^N\setminus C_t)=1\Leftrightarrow\rho^{in}=0\hbox{ a.e. on }Z_t\,.
$$
b) Under any one of the equivalent conditions in statement (a), one has $\rho(t)\ll\scrL^N$ with Radon-Nikodym derivative
$$
\rho(t,x):=\frac{d\rho(t)}{d\scrL^N}(x)=\sum_{y\in F_t^{-1}(\{x\})}\frac{\rho^{in}(y)}{J_t(y)}\quad\hbox{ for a.e. }x\in\bR^N\,.
$$
c) Under any one of the equivalent conditions in statement (a), the Borel probability measure $\mu(t)$ has a disintegration with respect to the Lebesgue
measure $\scrL^N$ on $\bR^N_x$ and the canonical projection $\Pi$ given by the formula
$$
\mu(t,x,\cdot)=\sum_{y\in F_t^{-1}(\{x\})}\frac{\rho^{in}(y)}{J_t(y)}\de_{\Xi_t(y,U^{in}(y))}\quad\hbox{ for a.e. }x\in\bR^N\,.
$$
\end{Thm}

\smallskip
Much less is known on the structure of the probability measures $\mu(t)$ and $\rho(t)$ when the equivalent conditions in statement (a) are not satisfied.
The available information is summarized in the following theorem.

\begin{Thm}\lb{T-DecompRhot}
Let $U^{in}$ be a continuous vector field on $\bR^N$ satisfying (\ref{Sublin}) and (\ref{CondLN1}). Let $t\in\bR$, and let $\rho^{in}$ be a probability density
on $\bR^N$. Let $\mu(t)$ and $\rho(t)$ be the Borel probability measures in (\ref{DefMut}), and let $\Lambda_t$ be the set defined in (\ref{DefLinLt}).

\smallskip
\noindent
a) For each $t\in\bR$, one has $\Supp(\mu(t))\subset\Lambda_t$.

\noindent
b) Let $\rho(t)=\rho_a(t)+\rho_s(t)$ be the Lebesgue decomposition of $\rho(t)$ with respect to the Lebesgue measure $\scrL^N$ on $\bR^N_x$, with
$$
\rho_a(t)\ll\scrL^N\quad\hbox{ and }\rho_s(t)\perp\scrL^N\,.
$$
Then $\rho_s(t)$ is carried by $C_t$ and 
$$
\rho_a(t)=F_t\#(\rho^{in}\indc_{P_t}\scrL^N)\,,\quad\hbox{ and }\rho_s(t)=F_t\#(\rho^{in}\indc_{Z_t}\scrL^N)\,.
$$
\end{Thm}

\smallskip
Applying statement (b) in Theorem \ref{T-Mu(t)=} shows that 
$$
\rho_a(t,x):=\frac{d\rho_a(t)}{d\scrL^N}(x)=\sum_{y\in F_t^{-1}(\{x\})}\frac{\rho^{in}(y)\indc_{P_t}(y)}{J_t(y)}\quad\hbox{ for a.e. }x\in\bR^N\,.
$$
Let $\mu^{in}_a$ be the monokinetic measure with momentum profile $U^{in}$ and density $\rho^{in}\indc_{P_t}$; by statement (c) of Theorem \ref{T-Mu(t)=},
$\mu_a(t):=\Phi_t\#\mu^{in}_a$ has a disintegration with respect to the Lebesgue measure $\scrL^N$ on $\bR^N_x$ and the canonical projection $\Pi$ given 
by the formula
$$
\mu_a(t,x,\cdot)=\sum_{y\in F_t^{-1}(\{x\})}\frac{\rho^{in}(y)\indc_{P_t}(y)}{J_t(y)}\de_{\Xi_t(y,U^{in}(y))}\quad\hbox{ for a.e. }x\in\bR^N\,.
$$

While the structure of $\mu_a(t)$ is well understood, the singular part $\rho_s(t)$ of the probability measure $\rho(t)$, and the phase space measure 
$\mu_s(t)=\mu(t)-\mu_a(t)$ are more difficult to characterize. The following theorem discusses the existence of atoms for the measure $\rho(t)$.

\begin{Thm}\lb{T-Atom}
Let $U^{in}$ be a continuous vector field on $\bR^N$ satisfying (\ref{Sublin}) and (\ref{CondLN1}). For each $t\in\bR$, consider the set
$$
A_t:=\{x\in\bR^N\hbox{ s.t. }\scrL^N(F_t^{-1}(\{x\})\cap Z_t)>0\}\,.
$$

\smallskip
\noindent
a) For each $t\in\bR$, one has $A_t\subset C_t$.

\noindent
b) Let $\rho^{in}$ be a probability density on $\bR^N$ such that $\rho^{in}>0$ a.e. on $Z_t$. Then
$$
\rho(t)(\{x\})>0\Leftrightarrow x\in A_t\,.
$$
c) The set $A_t$ is at most countable.
\end{Thm}

\smallskip
The singular part $\rho_s(t)$ in the Lebesgue decomposition of the probability measure $\rho(t)$ with respect to the Lebesgue measure $\scrL^N$ may 
indeed have atoms: see example \ref{E-Atom} in the next section. It may also happen that $\rho_s(t)$ is a nonzero diffuse measure: see example 
\ref{E-Diffuse} in section \ref{S-Expl} below.  

\subsection{On the classical limit of Schr\"odinger's equation (Problem D)}

In this section, we return to the original motivation for the present work, i.e. the classical limit of the Schr\"odinger equation. Specifically, we seek information
on the solution $\psi_\eps$ of the Schr\"odinger equation (\ref{ClassSchrod}) with WKB initial data (\ref{ScalIn}), in cases where Theorem \ref{T-WKB} cannot
be applied for lack of regularity of the initial amplitude $a^{in}$ and phase function $S^{in}$. Although lowering the regularity requirements on the potential
$V$ is not our main purpose in this paper, it is also worth noticing that the regularity requirements on $V$ in the theorem below are much less restrictive than
in Theorem \ref{T-WKB}.

We assume in this section that $V\in C^2(\bR^N)$ satisfies
\be\lb{HypVWeak}
|V(x)|+|\grad V(x)|=o(|x|)\quad\hbox{ and }|\grad^2V(x)|=O(1)\quad\hbox{ as }|x|\to\infty
\ee
Under this assumption, the Hamiltonian $H(x,\xi):=\tfrac12|\xi|^2+V(x)$ satisfies (\ref{HypH}) and, as explained above, generates a global flow 
$\Phi_t(x,\xi)=(X_t(x,\xi),\Xi_t(x,\xi))$ for all $t\in\bR$ and $x,\xi\in\bR^N$.

Assume further that
\be\lb{SA1}
\sup_{x\in\bR^N}\int_{\bR^N}\Gamma_\eta(x-y)V^-(y)dy\to 0\hbox{ as }\eta\to 0\quad\hbox{ if }N\ge 2
\ee
with 
$$
\Gamma_\eta(z)=\left\{\ba{}&\indc_{[0,\eta]}(|z|)|z|^{2-N}&&\hbox{ if }N\ge 3\,,\\&\indc_{[0,\eta]}(|z|)\ln(1/|z|)&&\hbox{ if }N=2\,,\ea\right.
$$
while
\be\lb{SA2}
\sup_{x\in\bR^N}\int_{x-1}^{x+1}V^-(y)dy<\infty\quad\hbox{ if }N=1\,.
\ee

\begin{Thm}\lb{T-WWKB}
Let $\psi_\eps$ the solution of the Schr\"odinger equation \eqref{ClassSchrod} with initial data (\ref{ScalIn}). Assume that $a^{in}\in L^2(\bR^N)$ satisfies 
the normalization $\|a^{in}\|_{L^2(\bR^N)}=1$ and that $S^{in}\in C^1(\bR^N)$ is such that $U^{in}:=\nabla S^{in}$ satisfies (\ref{Sublin}) and the 
regularity condition (\ref{CondLN1}). Let $t\in\bR$.

\smallskip
\noindent 
a) For each $\chi\in C_b(\bR^N)$ s.t. $\chi(F_t(y))\rho^{in}(y)=0$ for a.e. $y\in Z_t$, one has
$$
\int_{\bR^N}\chi(x)|\psi_\eps(t,x)|^2dx\to\int_{\bR^N}\chi(x)\sum_{y\in F^{-1}_t(\{x\})}\frac{|a^{in}|^2\indc_{P_t}}{J_t}(y)dx
$$
as $\eps\to 0$.

\smallskip
\noindent
b) For each $\chi\in C_b(\bR^N)$ s.t. $\chi(\Xi_t(y,\nabla S^{in}(y))\rho^{in}(y)=0$ for a.e. $y\in Z_t$, one has
$$
\int_{\bR^N}\chi(\xi)|\cF_\eps\psi_\eps(t,\xi)|^2d\xi
\to
\int_{\bR^N}\sum_{y\in F^{-1}_t(\{x\})}\chi(\Xi_t(y,\nabla S^{in}(y)))\frac{|a^{in}|^2\indc_{P_t}}{J_t}(y)dx
$$
as $\eps\to 0$, where the $\eps$-Fourier transform $\cF_\eps$ is defined as follows:
$$
\cF_\eps\Psi(\xi):=\frac1{(2\pi\eps)^{N/2}}\int_{L^2(\bR^N)}\Psi(x)e^{i\frac{x\xi}{\eps}}dx\,.
$$

\end{Thm}

\smallskip
In other words, even though the regularity assumptions on $a^{in}$, $S^{in}$ and $V$ do not allow us to apply Theorem \ref{T-WKB}, statements (a)-(b) in 
Theorem \ref{T-WWKB} provide information on $|\psi_\eps|^2$ and on $|\cF_\eps\psi_\eps|^2$ which is consistent with the approximation (\ref{WKBLimV}) 
by the WKB asymptotic solution (\ref{WKBSol})  in view of (\ref{SumWigner}). 

Theorem \ref{T-WWKB} is a consequence of the Lions-Paul Theorem \ref{T-Wigner}, and of our Theorem \ref{T-Mu(t)=} on the structure of the propagated 
Wigner measure. Whether more information on $\psi_\eps$ itself (instead of $|\psi_\eps|^2$ and $|\cF_\eps\psi_\eps|^2$) can be extracted from the approach
of the classical limit of quantum mechanics based on the Wigner transform seems to be an interesting open question.

\section{Main results II: Examples and counterexamples}\lb{S-Expl}

This section gathers various examples showing that our results in the previous section are sharp. All these examples are in space dimension $1$, so that the 
initial momentum profile $U^{in}$ is the gradient of the phase function
\be\lb{InPhase}
S^{in}(x)=\int_0^xU^{in}(z)dz\,.
\ee

\subsection{On the regularity condition (\ref{CondLN1}) on the momentum profile $U^{in}$}

The regularity condition (\ref{CondLN1}) on the initial momentum profile $U^{in}$ was chosen in order to apply the area formula of geometric measure theory. 
This formula is the key argument in the proof of statements (c)-(d)-(e)-(f) in Theorem \ref{T-NURough}, of statement (b) in Corollary \ref{C-NURough}) and of 
statements (b)-(c) in Theorem \ref{T-Mu(t)=}. 

If the Hamiltonian function $H$ satisfies the assumptions (\ref{HypH}), it generates a global, $C^1$ Hamiltonian flow $\Phi_t$ as explained in section
\ref{SS-DynInit}.  Let $U^{in}$ be a continuous vector field on $\bR^N$ and $\rho^{in}$ be a probability density on $\bR^N$, and let $\mu^{in}$ be the 
monokinetic measure with momentum profile $U^{in}$ and density $\rho^{in}$ (see Definition \ref{D-Monokin}). Then $\mu(t):=\Phi_t\#\mu^{in}$ is a well 
defined element of $C_b(\bR;w-\cP(\bR^N_x\times\bR^N_\xi))$, which is a solution of the Liouville equation of classical mechanics:
$$
\d_t\mu+\Div_x(\mu\grad_\xi H(x,\xi))-\Div_\xi(\mu\grad_xH(x,\xi))=0\,.
$$

This suggests the following questions: can one extend the validity of our Theorem \ref{T-Mu(t)=} on the structure of $\mu(t)$ to cases where $U^{in}$ does not
satisfy the regularity assumption (\ref{CondLN1})? In other words, is is still true that the disintegration of $\mu(t)$ is given by an a.e. finite sum of monokinetic
measures in the complement of the Lebesgue-negligible caustic fiber $C_t$?

The example below answers these questions in the negative. In this example, the Hamiltonian flow $\Phi_t$ is the free flow, the space dimension is $1$,
the initial momentum profile is continuous with compact support, and even of bounded variation, but not absolutely continuous. We prove that the caustic
fiber $C_t$ is not Lebesgue-negligible for some instant of time $t>0$.

\begin{Expl}\lb{E-CantorFunc}
Let $N=1$ and $H(x,\xi)=\tfrac12\xi^2$, so that $\Phi_t(x,\xi)=(x+t\xi,\xi)$. Let $K\subset[0,1]$ be the ternary Cantor set. We recall that $K$ satisfies
$\scrH^s(K)=1$ with $s=\ln2/\ln3$. Set
$$
U^{in}(z):=\left\{\ba{}&0&&\quad\hbox{ if }z\notin[0,1]\,,\\&\scrH^s([0,z]\cap K)-z\,,&&\quad\hbox{ if }z\in[0,1]\,.\ea\right.
$$
a) Then $U^{in}\in C_c(\bR)\cap BV(\bR)$, but the signed measure $(U^{in})'$ is not an element of $L^{1,1}(\bR)=L^1(\bR)$, so that (\ref{CondLN1}) is not 
satisfied.

\noindent
b) At $t=1$, the map $F_1$ is the Cantor function, given by the formula
$$
F_1(y):=\left\{\ba{}&y&&\quad\hbox{ if }y\notin[0,1]\,,\\&\scrH^s([0,y]\cap K)\,,&&\quad\hbox{ if }y\in[0,1]\,.\ea\right.
$$
In particular, $F_1\in C(\bR)$ and is increasing (and therefore $F_1\in BV_{loc}(\bR)$).

\noindent
c) The map $F_1$ is differentiable on $\bR\setminus K$, with $F_1'(y)=0$ for all $y\in[0,1]\setminus K$; besides, $F_1$ is not differentiable on $K$.

\noindent
d) The caustic fiber at time $t=1$ is $C_1=[0,1]$.
\end{Expl}

\smallskip
In other words, any initial wave function (\ref{ScalIn}) for the free Schr\"odinger equation with initial phase (\ref{InPhase}) leads to a caustic fiber of positive measure at time $t=1$.

\smallskip
Notice that, in this example, the set $Z_1=(F'_1)^{-1}(\{0\})=[0,1]\setminus K$ is an open set of $(0,1)$. As such, $Z_1$ is a countable union of open intervals
on which $F_1$ is a constant, and thus $F_1(Z_1)$ is at most countable and therefore Lebesgue negligible. In other words, the fact that $C_1$ is of
positive Lebesgue measure strongly depends upon the inclusion of the nondifferentiability set $E$ of $F_1$ (here $E=K$) in the definition (\ref{DefC}) of the
caustic fiber. If one seeks to extend our analysis of the propagation of monokinetic measures to the case of momentum profiles $U^{in}$ less regular than
those considered in the present paper, the choice of the Lebesgue negligible set used in the place of $E$ in the definition of the caustic fiber becomes crucial
in order to avoid caustic fibers of positive Lebesgue measure. This choice however does not have any effect on the measures $\mu(t)$ and $\rho(t)$, as 
explained above. These issues will be addressed elsewhere.

\subsection{On the structure of the singular measure $\rho_s(t)$}

As explained above, the singular part $\rho_s(t)$ in the Lebesgue decomposition of the measure $\rho(t)$ with respect to the Lebesgue measure $\scrL^N$ 
may have an atomic part, which can be constructed easily following statement (b) in Theorem \ref{T-Atom}. Here is an example.

\begin{Expl}\lb{E-Atom}
Let $N=1$ and $H(x,\xi)=\tfrac12\xi^2$, so that $\Phi_t(x,\xi)=(x+t\xi,\xi)$. Let $U^{in}\in C^\infty(\bR)$ be defined by the formula
$$
U^{in}(z):=-\int_0^zv(z)dz
$$
where $v\in C^\infty_c(\bR)$ is a bump function chosen so that
$$
\Supp(v)\subset(-1,1)\,,\quad v(z)=1\hbox{ for }|z|\le\tfrac12\,,\quad 0<v(z)<1\hbox{ for }|z|\in(\tfrac12,1)\,.
$$

For each $t\in\bR$, the map $F_t:\,\bR\to\bR$ is given by the formula $F_t(y)=y+tU^{in}(y)$. For $t=1$, one has $Z_1=[-\tfrac12,\tfrac12]$ and $F_1(Z_1)=\{0\}$.
Therefore $A_1=C_1=\{0\}$. Thus, if $\rho^{in}\in C^\infty(\bR)$ is a probability density with $\Supp(\rho^{in})\subset(-\tfrac12,\tfrac12)$, one has
$$
\rho(1)=F_1\#(\rho^{in}\scrL^1)=\de_0\,.
$$
\end{Expl}

\smallskip
Notice that the initial momentum profile $U^{in}$ and the initial density $\rho^{in}$ can be chosen as $C^\infty$ functions. In other words, the assumptions of 
Theorem \ref{T-WKB} are satisfied in this example (with potential $V=0$). Not much information can be gained from applying Theorem \ref{T-WKB} in this 
case, since the asymptotic WKB wave function $\Psi_\eps$ in this case would satisfy $\Psi_\eps(1,\cdot)=0$ on $\bR\setminus C_1$. Yet, the solution of the 
free Schr\"odinger equation for an initial wave function (\ref{ScalIn}) with initial phase function (\ref{InPhase}) and $\rho^{in}=(a^{in})^2$ where $a^{in}$ is 
a $C^\infty$ function with support in $(-\tfrac12,\tfrac12)$ is a semiclassical Lagrangian distribution, and Example \ref{E-Atom} can be formulated within the
formalism of Lagrangian distributions. (See for instance \S 2.2 and 4.2 in \cite{PaulUribe} for an account of this theory.)

\smallskip
Showing that the singular part $\rho_s(t)$ in the Lebesgue decomposition of the measure $\rho(t)$ with respect to the Lebesgue measure $\scrL^N$ may be
a nontrivial diffuse measure is less obvious. 

\begin{Expl}\lb{E-Diffuse}
Let $N=1$ and $H(x,\xi)=\tfrac12\xi^2$, so that $\Phi_t(x,\xi)=(x+t\xi,\xi)$. Pick $K$, a compact subset of $(0,1)\setminus\bQ$ such that $\scrL^1(K)\in(\tfrac12,1]$.
For each $k\ge 1$, there exists $U^{in}\in C^k_b(\bR)$ such that the map $F_t:\,\bR\to\bR$, defined by the formula $F_t(y)=y+tU^{in}(y)$ satisfies the following
conditions:

\smallskip
\noindent
a) for $t=1$, the map $F_1$ is increasing on $\bR$ and onto;

\noindent
b) one has $F_1'(y)>0$ for all $y<0$, all $y>1$ and all $y\in(0,1)\setminus K$, while $F_1'(y)=0$ for all $y\in K\cup\{0,1\}$, so that the caustic fiber at time $1$ is 
$C_1=F_1(K\cup\{0,1\})$.

\smallskip
Then, for any probability density $\rho^{in}$ on $\bR$ such that $\rho^{in}=0$ a.e. on $\bR\setminus K$ --- for instance, $\rho^{in}=\frac1{\scrL^1(K)}\indc_K$
--- the measure $\rho(1)=F_1\#(\rho^{in}\scrL^1)$ satisfies the following properties
$$
\rho(1)(\bR)=1\,,\quad\rho(1)\perp\scrL^1=0\,,\quad\hbox{ and }\rho(1)(\{x\})=0\hbox{ for all }x\in\bR\,.
$$
\end{Expl}

\smallskip
Notice that the initial profile $U^{in}$ can be chosen as a $C^k$ function with $k$ arbitrarily large. However the initial density $\rho^{in}$ is not smooth in this 
example, so that the regularity assumptions of Theorem \ref{T-WKB} are not satisfied. In the language of wave functions, this example corresponds to the
propagation (under the dynamics of the free Schr\"odinger equation) of a WKB ansatz (\ref{ScalIn}) with initial phase function (\ref{InPhase}) of class $C^k$ 
with $k$ arbitrarily large. Yet the usual tools of semiclassical analysis (such as the formalism of Lagrangian distributions) do not apply to this case, as the 
initial amplitude $a^{in}=\sqrt{\rho^{in}}$ cannot be smooth, or even continuous, in this example.

\subsection{On the Hausdorff dimensions of the caustic fiber and of the support of $\rho_s(t)$}

The singular part $\rho_s(t)$ in the Lebesgue decomposition of the probability measure $\rho(t)$ with respect to $\scrL^N$ is carried by the caustic fiber $C_t$ 
which is a Lebesgue-negligible set. The next example shows that $\Supp(\rho(t))$ can be of arbitrary Hausdorff dimension.

We briefly recall the following construction (see \cite{Falco} on p. 15) which generalizes the construction of the ternary Cantor set.

Let $\th\in(0,\tfrac12)$. Set $E_0:=[0,1]$, let $E_1:=[0,\th]\cap[1-\th,1]$, and obtain $E_{n+1}$ from $E_n$ by removing the open segment of proportion $1-2\th$ 
from the center of each connected component in $E_n$. Therefore $E_n$ is the union of $2^n$ closed segments of length $\th^n$. In other words,
$$
[0,1]\setminus E_n=\bigcup_{1\le k\le 2^{m-1}\atop 1\le m\le n}I_{m,k}
$$
where
$$
I_{m,k}:=(a_{m,k}-r_m,a_{m,k}+r_m)\quad\hbox{ with }r_m:=\tfrac12(1-2\th)\th^{m-1}\,.
$$
Then 
$$
K(\th)=\bigcap_{n\ge 0}E_n
$$
is a compact subset of $[0,1]$ such that $\scrH^{s}(K(\th))=1$ with $s=\ln 2/\ln(1/\th)$.

\begin{Expl}\lb{E-HausDim}
Let $N=1$ and $H(x,\xi)=\tfrac12\xi^2$, so that $\Phi_t(x,\xi)=(x+t\xi,\xi)$. Let $s\in(0,1)$, and let $\th=2^{-1/s}$. Set
$$
\Om(\th):=\bigcup_{1\le k\le 2^{m-1}\atop m\ge 1}J_{m,k}
$$
with
$$
J_{m,k}:=(a_{m,k}-\th r_m,a_{m,k}+\th r_m)\,,
$$
and let $\tilde K(\th)=[0,1]\setminus\Om(\th)$. Define the initial momentum profile $U^{in}$ by the formula
$$
U^{in}(y):=\left\{\ba{}&\frac1\th\scrL^1(\Om(\th)\cap(0,y))-y&&\quad\hbox{ if }y\in[0,1]\,,\\ &0&&\quad\hbox{ if }y\notin[0,1]\,,\ea\right.
$$
so that $U^{in}\in\Lip(\bR)$ and $\Supp(U^{in})\subset[0,1]$. Let $F_t(y)=y+tU^{in}(y)$ for all $t,y\in\bR$. Then

\smallskip
\noindent
a) for $t=1$, the function $F_1$ is increasing on $(-\infty,0)\cup\Om(\th)\cup(1,+\infty)$ and nondecreasing on $\bR$;

\noindent
b) for $t=1$, one has $F_1((-\infty,0)\cup\Om(\th)\cup(1,+\infty))=\bR\setminus K(\th)$ and $F_1(\tilde K(\th))=K(\th)$;

\noindent
c) for $\rho^{in}=\frac1{1-\th}\indc_{\tilde K(\th)}$, let $\mu^{in}$ be the monokinetic measure with momentum profile $U^{in}$ and density $\rho^{in}$,
and let $\mu(t)=\Phi_t\#\mu^{in}$. Then
$$
\mu(1)=\frac1{1-\th}\sum_{m\ge 1}\sum_{k=1}^{2^{m-1}}(\de_{a_{m,k}-r_m}\otimes\indc_{(-(1-\th)r_m,0)}+\de_{a_{m,k}+r_m}\otimes\indc_{(0,(1-\th)r_m)})\,;
$$
c) under the same assumptions as in statement (c),
$$
\rho(1)=F_1\#\rho^{in}=\Pi\#\mu(1)=\tfrac12(1-2\th)\sum_{m\ge 1}\th^{m-1}\sum_{k=1}^{2^{m-1}}(\de_{a_{m,k}-r_m}+\de_{a_{m,k}+r_m})\,,
$$
so that
$$
\rho(1)(\bR)=1\,,\quad\Supp(\rho(1))=K(\th)\,,\quad\hbox{ and }\rho(1)\perp\scrL^1\,;
$$
e) for $t=1$, one has $C_1\cap(0,1)=K(\th)\cap(0,1)$, so that $\scrH^s(C_1)=\scrH^s(K(\th))=1$ with $s=\ln2/\ln(1/\th)$.
\end{Expl}

This example corresponds to the propagation (under the dynamics of the free Schr\"odinger equation) of a WKB type initial wave function (\ref{ScalIn}), with
initial phase (\ref{InPhase}), leading to a Wigner measure at time $t=1$ which is carried by the set
\be\lb{ConcentrSet}
\bigcup_{m\ge 1}\!\left(\bigcup_{k=1}^{2^m-1}\{a_{m,k}-r_m\}\!\right)\times(-(1-\th)r_m,0)
\cup\bigcup_{m\ge 1}\!\left(\bigcup_{k=1}^{2^m-1}\{a_{m,k}+r_m\}\!\right)\times(0,(1-\th)r_m).
\ee
Since the solution of the Cauchy problem for the free Schr\"odinger equation is given by an explicit formula, this construction provides an example of a wave function 
whose Wigner is concentrated precisely on the set (\ref{ConcentrSet}) above.  

Notice that $\rho_s(1)$ is carried by the countable set 
$$
\{a_{m,k}\pm r_m\,|\,1\le k\le 2^{m-1}\,,\,\,m\ge 1\}
$$
even though $\Supp(\rho_s(1))=K(\th)$ is of Hausdorff dimension $s=\ln2/\ln(1/\th)$.

Example \ref{E-HausDim} is vaguely reminiscent of Example \ref{E-Atom}: it corresponds to an infinite accumulation of wave functions as in Example \ref{E-Atom}, 
except that the initial amplitude cannot be chosen smooth, or even continuous. The concentration of the Wigner measure on the set (\ref{ConcentrSet}) at time 
$t=1$ is produced by the combination of oscillations in the WKB initial data (\ref{ScalIn})-(\ref{InPhase}) at the same scale as the characteristic scale $\eps$ of the 
Wigner transform with the dynamics of the free Schr\"odinger equation. The importance of fast oscillations in this concentration phenomenon can be seen in the 
fact that the support of the Wigner measure at time $t=1$ is not included in the null section $\bR^N_x\times\{0\}$ of the phase space $\bR^N_x\times\bR^N_\xi$. 

At this point, we return to the remarks on our definition of caustic fiber following (\ref{DefCt}) and Example \ref{E-CantorFunc}. Examples \ref{E-CantorFunc} and 
(\ref{E-HausDim}) show that the regularity in assumption (\ref{CondLN1}) corresponds to a threshold in the size of the caustic fiber defined in (\ref{DefCt}), at 
least in the case of space dimension $1$. Indeed, if $U^{in}$ satisfies (\ref{CondLN1}), the caustic fiber is Lebesgue negligible, but can be of arbitrary Hausdorff 
dimension in $(0,1)$, while if $U^{in}$ is of bounded variation but does not satisfy (\ref{CondLN1}), the caustic fiber can be of positive Lebesgue measure.
Thus the definition (\ref{DefCt}) of the caustic fiber is consistent with the regularity condition (\ref{CondLN1}). Of course, as explained above, our choice in
the definition of the caustic fiber does not have any effect the propagation of the monokinetic measure.

\section{ The Hamiltonian flow $\Phi_t$}\lb{S-HAMFLOW}

In this section, we have collected the properties of the Hamiltonian flow $\Phi_t$ defined in (\ref{DefPhi}) which are used throughout the paper. 

\begin{Lem}\lb{L-HamFlow}
The map $(t,x,\xi)\mapsto\Phi_t(x,\xi)$ is of class $C^1$ on $\bR\times\bR^N\times\bR^N$. For each $\eta>0$, there exists $C_\eta>0$ such that
$$
\sup_{|t|\le T}|X_t(x,\xi)-x|\le C_\eta(1+|\xi|)+\eta|x|
$$
for each $x,\xi\in\bR^N$. Moreover
$$
|D\Phi_t(x,\xi)-\Id_{\bR^N\times\bR^N}|\le e^{\ka|t|}-1
$$
for all $t\in\bR$ and each $x,\xi\in\bR^N$.
\end{Lem}

\begin{proof}
The existence a,d regularity of the flow on its domain of definition follows from the classical Cauchy-Lipschitz theory. 

By (\ref{HypH}), one also has the following a priori estimates, with the notation (\ref{DefPhi}). 

First
$$
|X_t(x,\xi)-x|\le\int_0^t|\grad_\xi H(\Phi_s(x,\xi))|ds\le\ka t+\ka\int_0^t|\Xi_s(x,\xi)|ds
$$
and
$$
\ba
|\Xi_t(x,\xi)|&\le|\xi|+\int_0^t|\grad_xH(\Phi_s(x,\xi))|ds
\\
&\le|\xi|+\ka\int_0^t|\Xi_s(x,\xi)|ds+\int_0^th(X_s(x,\xi))ds\,.
\ea
$$
By Gronwall's inequality, for all $0\le s\le t$
\be\lb{Xi<}
|\Xi_s(x,\xi)|\le\left(|\xi|+\int_0^th(X_\tau(x,\xi))d\tau\right)e^{\ka s}\,,
\ee
so that
$$
\ba
|X_t(x,\xi)-x|&\le\ka t+\ka\int_0^te^{\ka s}ds\left(|\xi|+\int_0^th(X_\tau(x,\xi))d\tau\right)
\\
&\le\ka t+e^{\ka t}\left(|\xi|+\int_0^th(X_\tau(x,\xi))d\tau\right)\,.
\ea
$$
Since $h$ is sublinear at infinity, we have, for every $R>0$
\be\lb{h<}
h(r)\le\indc_{[0,R]}(r)\sup_{0\le r\le R}h(r)+\indc_{(R,+\infty)}(r)r\sup_{r>R}\frac{h(r)}{r}\eta\le M_R+rm_R\,,
\ee
where 
$$
M_R=\sup_{0\le r\le R}h(r)\quad\hbox{ and }m_R=\sup_{r>R}\frac{h(r)}{r}\,,
$$
so that
$$
m_R\to 0\quad\hbox{ as }R\to+\infty\,.
$$
Therefore
$$
\ba
|X_t(x,\xi)-x|&\le(\ka+M_Re^{\ka t})t+e^{\ka t}|\xi|+m_Re^{\ka t}\int_0^t|X_s(x,\xi)|ds
\\
&\le(\ka+M_Re^{\ka t})t+e^{\ka t}|\xi|+m_Re^{\ka t}|x|+m_Re^{\ka t}\int_0^t|X_s(x,\xi)-x|ds\,.
\ea
$$
By Gronwall's inequality, 
\be\lb{X<}
\ba
|X_t(x,\xi)-x|&\le((\ka+M_Re^{\ka t})t+e^{\ka t}|\xi|+m_Re^{\ka t}|x|)e^{tm_Re^{\ka t}}
\\
&\le\ka te^{tm_Re^{\ka t}}+M_Rte^{t(\ka+m_Re^{\ka t})}+|\xi|e^{t(\ka+m_Re^{\ka t})}+m_R|x|e^{t(\ka+m_Re^{\ka t})}\,.
\ea
\ee
The same estimates hold for $-T\le t\le 0$ after substituting $|t|$ to $t$.

In view of (\ref{X<})-(\ref{Xi<}), for each $(x,\xi)\in\bR^N\times\bR^N$, the trajectory $(x,\xi)\mapsto\Phi_t(x,\xi)$ cannot escape to infinity in finite time,
and is therefore globally defined. 

Besides, since $m_R\to 0$ as $R\to+\infty$, the estimate (\ref{X<}) obviously implies the first inequality in the lemma with 
$$
\eta:=m_Re^{T(\ka+m_Re^{\ka T})}\quad\hbox{ and }C_\eta:=(1+\ka T+M_RT)e^{T(\ka+m_Re^{\ka T})}\,.
$$

Since $H\in C^2(\bR^N\times\bR^N)$, the map $(t,x,\xi)\mapsto\Phi_t(x,\xi)$ is of class $C^1$ on its domain of definition $\bR\times\bR^N\times\bR^N$. 
Differentiating the Hamilton equations with respect to the initial condition, one finds that
$$
\left\{
\ba
{}&\dot{DX}_t=+\grad^2_{x,\xi}H(\Phi_t)\cdot DX_t+\grad^2_{\xi,\xi}H(\Phi_t)\cdot D\Xi_t\,,
\\
&\dot{D\Xi}_t\,=-\grad^2_{x,x}H(\Phi_t)\cdot DX_t-\grad^2_{x,\xi}H(\Phi_t)\cdot D\Xi_t\,,
\ea
\right.
$$
so that
$$
|D\Phi_t-\Id_{\bR^N\times\bR^N}|\le\ka\int_0^{|t|}|D\Phi_s|ds\,.
$$
The second inequality in the lemma follows from Gronwall's inequality.
\end{proof}

\section{Proofs of Theorems \ref{T-NURough} and \ref{T-NUC1} and of Corollary \ref{C-NURough}}\lb{S-N}

We shall need the following more or less classical topological argument.

\begin{Lem}\lb{L-Top}
Let $g:\,\bR^N\to\bR^N$ be a continuous map satisfying the following condition: for some $R>0$
$$
(g(x)|x)>0\quad\hbox{ for all }x\in\bR^N\hbox{ such that }|x|=R\,.
$$
Then 

\smallskip
\noindent
a) there exists $x\in\bR^N$ such that $|x|\le R$ and $g(x)=0$;

\noindent
b) if $g$ is of class $C^1$ on $\bR^N$ and $0$ is a regular value of $g$, then $g^{-1}(\{0\})\cap B(0,R)$ is finite and $\#(g^{-1}(\{0\})\cap B(0,R))$ is odd.
\end{Lem}

\begin{proof}
Consider the homotopy $G\in C([0,1]\times\bR^N;\bR^N)$ defined by 
$$
G(t,x)=tx+(1-t)g(x)\,.
$$
One has
$$
G(t,x)\not=0\quad\hbox{ whenever }t\in[0,1]\hbox{ and }|x|=R\,.
$$
Indeed, $G(1,x)=x\not=0$ if $|x|=R>0$; besides, if $t\in[0,1[$ and $G(t,x)=0$, one has 
$$
g(x)=-\frac{t}{1-t}x\quad\hbox{ so that }(g(x)|x)=-\frac{t}{1-t}|x|^2=-\frac{t}{1-t}R^2<0
$$
for all $x\in\bR^N$ such that $|x|=R$, which contradicts our assumption. 

By the homotopy invariance of the degree (see  Properties 7, 8 and Theorem 12.7 in chapter 12, \S A of \cite{Smoller})
$$
d(g,B(0,R),0)=d(I,B(0,R),0)=1\,.
$$
This implies a). 

Moreover, if $g$ is of class $C^1$ on $\bR^N$ and $0$ is a regular value of $g$, all the elements of $g^{-1}(\{0\})$ are isolated points by the implicit function 
theorem, so $g^{-1}(\{0\})\cap\overline{B(0,R)}$ is finite. Besides (see  Property 2 in chapter 12, \S A of \cite{Smoller})
$$
d(g,B(0,R),0)=\sum_{x\in g^{-1}(\{0\})\cap B(0,R)}\Sign(\Det Dg(x))=1\,.
$$
Therefore, there exists an integer $m\in\bN$ such that
$$
\ba
\#\{x\in B(0,R)\,|\,g(x)=0\hbox{ and }\Det(Dg(x))>0\}&=m+1
\\
\#\{x\in B(0,R)\,|\,g(x)=0\hbox{ and }\Det(Dg(x))<0\}&=m
\ea
$$
so that $\#(g^{-1}(\{0\})\cap B(0,R))=2m+1$, which proves b).
\end{proof}

\subsection{Proof of Theorem \ref{T-NURough}}

The map $F_t$ is continuous, being the composition of the continuous maps $y\mapsto (y,U^{in}(y))$ and $(x,\xi)\mapsto X_t(x,\xi)$. By the first inequality in 
Lemma \ref{L-HamFlow} and the condition (\ref{Sublin}) on $U^{in}$, for each $\eta>0$, one has
$$
\varlimsup_{|y|\to+\infty}\sup_{|t|\le T}\frac{|F_t(y)-y|}{|y|}\le\eta\,,
$$
which is precisely the estimate in statement (a). 

This estimate implies 
\be\lb{Deg<}
(F_t(y)-x|y)=|y|^2+o(|y|^2)\quad\hbox{ as }|y|\to+\infty\,,
\ee
so that $F_t$ is onto by applying statement (a) in Lemma \ref{L-Top} to the map $g:\,y\mapsto F_t(y)-x$. 

It also implies
$$
|F_t(y)|\to+\infty\quad\hbox{ as }|y|\to+\infty
$$
so that $F_t$ is proper. This establishes statement (a).

By the second estimate in Lemma \ref{L-HamFlow}
\be\lb{BdDFty}
|D_xX_t(y,U^{in}(y))+D_\xi X_{t}(y,U^{in}(y))\cdot DU^{in}(y)|\le e^{\ka|t|}+(e^{\ka|t|}-1)|DU^{in}(y)|
\ee
so that
\be\lb{BdJ}
\ba
J_t(y)&=|\Det(D_xX_t(y,U^{in}(y))+D_\xi X_{t}(y,U^{in}(y))\cdot DU^{in}(y))|
\\
&\le e^{N\ka|t|}(1+(1-e^{-\ka|t|})|DU^{in}(y)|)^N
\ea
\ee
by Hadamard's inequality. Since $U^{in}$ satisfies (\ref{CondLN1}) and since $L^{N,1}(B(0,R))\subset L^N(B(0,R))$ for each $R>0$, this inequality implies 
statement (b).

Since the map $F_t$ is proper by statement (a), the set $K_{t,m}=F_t^{-1}(\overline{B(0,m)})$ is compact for each $m\in\bN^*$. Applying the area formula 
(Theorem 3.4 in \cite{Maly3} and Theorem A in \cite{KKM}), one has
$$
\int_{\bR^N}\#(F_t^{-1}(\{x\})\cap K_{t,m})dx=\int_{K_{t,m}}J_t(y)dy<+\infty\,.
$$
Therefore $\# F_t^{-1}(\{x\})<\infty$ for a.e. $x\in\overline{B(0,m)}$, i.e. for all $x\in\overline{B(0,m)}\setminus E_m$ with $\scrL^N(E_m)=0$. Thus
$$
\# F_t^{-1}(\{x\})<\infty\hbox{ for all }x\in\bR^N\setminus\bigcup_{m\ge 1}E_m
$$
and 
$$
\scrL^N\left(\bigcup_{m\ge 1}E_m\right)\le\sum_{m\ge 1}\scrL^N(E_m)=0\,,
$$
which is statement (c).

Next we prove statement (d). Let $R>0$; applying again the area formula shows that 
$$
\int_{\bR^N}\cN_R(t,x)dx=\int_{\overline{B(0,R)}}J_t(y)dy\,.
$$
By the Bienaym\'e-Chebyshev inequality, for each $n\ge 1$
$$
\scrL^N\left(\{x\in\bR^N\hbox{ s.t. }\cN_R(t,x)\ge n\}\right)\le\frac1n\int_{\overline{B(0,R)}}J_t(y)dy\,.
$$
In view of (\ref{BdJ}), one has
$$
\ba
\int_{\overline{B(0,R)}}J_t(y)dy&\le e^{N\ka|t|}\int_{\overline{B(0,R)}}(1+(1-e^{-\ka|t|})|DU^{in}(y)|)^Ndy
\\
&\le e^{N\ka|t|}\|1+|DU^{in}(y)|\|_{L^N(B(0,R))}^N\,.
\ea
$$
With the inequality above, this is precisely the estimate in statement (d).

Let $m\in\bN^*$ and $B_{t,m}=(Z_t\cup E)\cap K_{t,m}$ with $Z_t$ as in (\ref{DefZP}) and $K_{t,m}$ as in the proof of statement (c). Thus the set $B_{t,m}$ is 
measurable and bounded. Applying the area formula as in the proof of statement (c) shows that
$$
\int_{\bR^N}\#(F_t^{-1}(\{x\})\cap B_{t,m})dx=\int_{B_{t,m}}J_t(y)dy=0
$$
since $J_t(y)=0$ for all $y\in B_{t,m}\setminus E$, i.e. for a.e. $y\in B_{t,m}$. By the Bienaym\'e-Chebyshev inequality,
$$
\scrL^N(\{x\in\bR^N\hbox{ s.t. }\#(F_t^{-1}(\{x\})\cap B_{t,m})\ge 1\})=0
$$
and therefore
$$
\ba
\scrL^N(C_t)&=\scrL^N(\{x\in\bR^N\hbox{ s.t. }F_t^{-1}(\{x\})\cap(Z_t\cup E)\not=\varnothing\})
\\
&=\scrL^N(\{x\in\bR^N\hbox{ s.t. }\#(F_t^{-1}(\{x\})\cap(Z_t\cup E))\ge 1\})
\\
&\le\sum_{m\ge 1}\scrL^N(\{x\in\bR^N\hbox{ s.t. }\#(F_t^{-1}(\{x\})\cap B_{t,m})\ge 1\})\,,
\ea
$$
which is precisely statement (e).

Next consider the continuous map 
$$
F:\,[-T,T]\times\bR^N\ni(t,y)\mapsto F(t,y)\in\bR^N\,.
$$
In view of statement (a), $|F(t,y)|\to\infty$ as $|y|\to+\infty$ uniformly in $t\in[-T,T]$. Therefore, the set $K_m:=F^{-1}(\overline{B(0,m)})$ is compact for 
each $m\in\bN^*$. Besides, for each $t\in[-T,T]$ and each $y\in\bR^N\setminus E$, the Jacobian $DF(t,y)$ is the column-wise partitioned matrix
$$
DF(t,y)=[V(t,y)\,,\,M(t,y)]\,,
$$
with
$$
V(t,y)=\grad_\xi H(\Phi_t(y,U^{in}(y)))
$$
and
$$
M(t,y):=D_xX_t(y,U^{in}(y))+D_\xi X_t(y,U^{in}(y))DU^{in}(y)\,.
$$
Therefore,
$$
DF(t,y)DF(t,y)^T=V(t,y)V(t,y)^T+M(t,y)M(t,y)^T
$$
so that, by the co-area formula (Theorem 1.3 in \cite{Maly})
$$
\ba
\int_{\bR^N}\scrH^1&(F^{-1}(\{x\})\cap K_m)dx
\\
&=\int_{K_m}\sqrt{\Det(V(t,y)V(t,y)^T+M(t,y)M(t,y)^T)}dtdy\,.
\ea
$$
By Lemma \ref{L-HamFlow}, $(t,x,\xi)\mapsto\Phi_t(x,\xi)$ is of class $C^1$ on $\bR\times\bR^N\times\bR^N$, so that the map $(t,y)\mapsto V(t,y)$ is
continuous on $\bR\times\bR^N$, and therefore bounded on the compact $K_m$. On the other hand, by (\ref{BdDFty})
$$
\sup_{|t|\le T}|M(t,y)|\le e^{\ka T}+(e^{\ka T}-1)|DU^{in}(y)|\in L^N_{loc}(\bR^N)\,,
$$
since $U^{in}$ satisfies (\ref{CondLN1}). Denoting
$$
K'_m:=\{y\in\bR^N\,|\,\hbox{ there exists }t\in[-T,T]\hbox{ s.t. }(t,y)\in K_m\}
$$
which is compact in $\bR^N$ (being the projection of the compact $K_m$ on the second factor in $\bR\times\bR^N$), one has
$$
\ba
\|VV^T+MM^T\|^{N/2}_{L^{N/2}(K_R)}&\le 2^{N/2-1}\|V\|^N_{L^\infty(K_R)}\scrL^{N+1}(K_R)+2^{N/2}T\|M\|^N_{L^N(K'_R)}
\\
&<\infty\,.
\\
\ea
$$
Therefore, $\scrH^1(F^{-1}(\{x\})\cap K_m)<+\infty$ is finite for a.e. $x\in\overline{B(0,m)}$ for each $m\ge 1$, i.e. for all $x\in\overline{B(0,m)}\setminus E'_m$ 
with $\scrL^N(E'_m)=0$. Since this is true for all $m\in\bN^*$, one concludes that
$$
\scrH^1(F^{-1}(\{x\}))<+\infty\hbox{ for all }x\in\bR^N\setminus\left(\bigcup_{m\ge 1}E'_m\right)\,,
$$
and
$$
\scrL^N\left(\bigcup_{m\ge 1}E'_m\right)\le\sum_{m\ge 1}\scrL^N(E'_m)=0\,,
$$
which is statement (f). The proof is complete.

\subsection{Proof of Corollary \ref{C-NURough}}

By statement (a) of Theorem \ref{T-NURough}, $M_T(R)\to 0$ as $R\to\infty$. Therefore there exists $R_T^*>0$ such that $M_T(R_T^*)<\tfrac12$. Since the 
function $M_T$ is nonincreasing by construction, $M_T(R)\le\tfrac12$ for all $R\ge R_T^*$. Therefore, if $R\ge R_T^*$, then 
$$
|y|\ge R\Rightarrow||F_t(y)|-|y||\le|F_t(y)-y|\le\tfrac12|y|\Rightarrow|F_t(y)|\ge\tfrac12|y|
$$
for all $t\in[-T,T]$, so that
$$
F_t(\bR^N\setminus\overline{B(0,R)})\subset\bR^N\setminus\overline{B(0,\tfrac12R)}\,.
$$
In other words, if $t\in[-T,T]$ and if $R\ge R_T^*$, then
$$
F_t^{-1}(\{x\})\subset\overline{B(0,R)}\hbox{ for all }x\in\overline{B(0,\tfrac12R)}\,,
$$
so that
$$
|t|\le T\hbox{ and }|x|\le\tfrac12R\Rightarrow\cN(t,x)=\cN_R(t,x)\,.
$$
This proves statement (a).

Thus, for each $t\in[-T,T]$,
$$
\ba
\scrL^N\left(\{x\in\overline{B(0,\tfrac12R)}\hbox{ s.t. }\cN(t,x)\ge n\}\right)&
\\
=
\scrL^N\left(\{x\in\overline{B(0,\tfrac12R)}\hbox{ s.t. }\cN_R(t,x)\ge n\}\right)&
\\
\le
\scrL^N\left(\{x\in\bR^N\hbox{ s.t. }\cN_R(t,x)\ge n\}\right)&\,,
\ea
$$
and we conclude by statement (d) in Theorem \ref{T-NURough}.

\subsection{Proof of Theorem \ref{T-NUC1}}

Since $U^{in}$ is of class $C^1$ on $\bR^N$, the map $F_t$ is of class $C^1$ on $\bR^N$ for each $t\in\bR$, being the composition of the $C^1$ maps 
$(x,\xi)\mapsto X_t(x,\xi)$ and $y\mapsto(y,U^{in}(y))$. In particular, the nondifferentiability set in (\ref{DefE}) is $E=\varnothing$ so that the caustic fiber $C_t$ 
reduces to
$$
C_t=\{x\in\bR^N\hbox{ s.t. }F_t^{-1}(\{x\})\cap Z_t\not=\varnothing\}=F_t(Z_t)\,.
$$
Since $F_t$ is of class $C^1$ on $\bR^N$, the absolute value of its Jacobian determinant $J_t$ is continuous on $\bR^N$, and therefore $Z_t=J_t^{-1}(\{0\})$
is closed in $\bR^N$. By statement (a) in Theorem \ref{T-NURough}, the continuous map $F_t$ is proper; therefore $C_t=F_t(Z_t)$ is closed in $\bR^N$, being 
the image of a closed set by a continuous and proper map on $\bR^N$. This proves statement (a).

Let $t\in\bR$. For each $x\in\bR^N$, the set $F_t^{-1}(\{x\})$ is compact since $F_t$ is proper. If moreover $x\in\bR^N\setminus C_t$, all the solutions of the 
equation $F_t(y)-x=0$ are isolated by the implicit function theorem. Therefore the set $F_t^{-1}(\{x\})$ is finite for each $x\in\bR^N\setminus C_t$. Besides, the
implicit function theorem implies that the integer-valued counting function $\cN$ is a locally constant function of $(t,x)\in\bR\times\bR^N\setminus C$. Thus
the counting function $\cN$ is constant on each connected component of $\bR\times\bR^N\setminus C$. This completes the proof of statement (b).

Let $j\in\bN^*$, and let $\Om$ be a connected component of $\bR\times\bR^N\setminus C$ such that $\cN(t,x)=n\ge j$ for all $(t,x)\in\Om$. Then the implicit 
function theorem implies that, for all $(t,x)\in\Om$, the set of solutions $y$ of the equation $F_t(y)-x=0$ takes the form $\{y_k(t,x)\,|\,1\le k\le n\}$, and one has
$y_k\in C^1(\Om)$ for all $k=1,\ldots,n$. This proves statement (c).

Assume $\inf\{t>0\,|\,C_t\not=\varnothing\}=0$. Then, there exists $(t_n,x_n,y_n)$ such that 
$$
t_n\to 0^+\,,\quad F_{t_n}(y_n)=x_n\,,\quad\hbox{ and }J_{t_n}(y_n)=0\,.
$$
Assume that some subsequence $y_{n_k}$ of the sequence $y_n$ is bounded. Up to further extraction of a subsequence, one can assume that $y_{n_k}\to y$, 
so that $0=J_{t_{n_k}}(y_{n_k})\to J_0(y)$. But since $F_0=\Id_{\bR^N}$, one has $J_0(y)=1$. Therefore $|y_n|\to+\infty$. By the second inequality in
Lemma \ref{L-HamFlow}
$$
\ba
{}&|D_xX_{t_n}(y_n,U^{in}(y_n))-\Id_{\bR^N}|\le e^{\ka|t_n|}-1\,,
\\
&|D_\xi X_{t_n}(y_n,U^{in}(y_n))\cdot DU^{in}(y_n)|=O\left(e^{\ka|t_n|}-1\right)\,,
\ea
$$
so that
$$
\ba
0=J_{t_n}(y_n)&=|\Det(D_xX_{t_n}(y_n,U^{in}(y_n))+D_\xi X_{t_n}(y_n,U^{in}(y_n))\cdot DU^{in}(y_n))|
\\
&\to|\Det(\Id_{\bR^N})|=1\quad\hbox{ as }n\to\infty\,.
\ea
$$
Thus the assumption $t_n\to 0$ leads to a contradiction. Therefore, 
$$
\inf\{t>0\,|\,C_t\not=\varnothing\}=b>0\,.
$$
By the same token, 
$$
\sup\{t<0\,|\,C_t\not=\varnothing\}=a<0\,.
$$
Thus $(a,b)\times\bR^N$ is contained in the connected component of $\{0\}\times\bR^N$ in $\bR\times\bR^N\setminus C$. Since $F_0=\Id_{\bR^N}$, one has 
$\cN(0,x)=1$ for all $x\in\bR^N$, and since $\cN$ is constant on each connected component of $\bR\times\bR^N\setminus C$, one concludes $\cN=1$ on the
strip $(a,b)\times\bR^N$, which proves statement (d).

As for statement (e), it follows from the inequality (\ref{Deg<}) implied by statement (a) in Theorem \ref{T-NURough} and from statement (b) in Lemma \ref{L-Top}
applied to the map $g:\,y\mapsto F_t(y)-x$. Indeed, $0$ is a regular value of this map since it is assumed that $x\in\bR^N\setminus C_t$. The proof is complete.

\section{Proofs of Theorems \ref{T-Mu(t)=}, \ref{T-DecompRhot} and \ref{T-Atom}}\lb{S-Mu}

\subsection{Proof of Theorem \ref{T-Mu(t)=}}

For each $t\in\bR$, the Borel measure $\rho(t)$ is a probability measure on $\bR^N$, so that $\rho(t)(C_t)=0$ if and only if $\rho(t)(\bR^N\setminus C_t)=1$.
Next, observe that
\be\lb{TranspRho}
\rho(t)=\Pi\#\mu(t)=\Pi\#(\Phi_t\#\mu^{in})=X_t\#\mu^{in}=F_t\#(\rho^{in}\scrL^N)
\ee
since $\mu^{in}$ is a monokinetic measure with density $\rho^{in}$ and momentum profile $U^{in}$ while $F_t(y)=X_t(y,U^{in}(y))$. Thus 
$$
\rho(t)(C_t)=\rho^{in}\scrL^N(F_t^{-1}(C_t))=\int_{Z_t}\rho^{in}(y)dy=0
$$
if and only if $\rho^{in}=0$ a.e. on $Z_t$, which proves statement (a).

Assume that $\rho^{in}=0$ a.e. on $Z_t$, and consider the measurable function defined a.e. on $\bR^N$ by the formula
\be\lb{Defb}
b(y):=\left\{\ba{}&\frac{\rho^{in}(y)}{J_t(y)}&&\qquad\hbox{ for a.e. }y\in P_t\,,\\ &0&&\qquad\hbox{ for each }y\notin P_t\,.\ea\right.
\ee
With this definition, one has
$$
\rho^{in}\scrL^N=bJ_t\scrL^N\,.
$$
In particular $\rho^{in}\scrL^N\ll J_t\scrL^N$, so that the class of the measurable function $b$ modulo equality $J_t\scrL^N$-a.e. is the unique element of 
$L^1(\bR^N;J_t\scrL^N)$ such that the equality above holds, by the Radon-Nikodym theorem (Theorem 6.10 in \cite{Rudin}).

Let $\chi\in C_c(\bR^N)$. By (\ref{TranspRho}) and the area formula (see Theorem 3.4 in \cite{Maly3} and Theorem A in \cite{KKM})
$$
\ba
\la\rho(t),\chi\ra=\la F_t\#(\rho^{in}\scrL^N),\chi\ra&=\int_{\bR^N}\chi(F_t(y))\rho^{in}(y)dy
\\
&=\int_{\bR^N}\chi(F_t(y))b(y)J_t(y)dy
\\
&=\int_{\bR^N}\left(\sum_{y\in F^{-1}_t(\{x\})}b(y)\right)\chi(x)dx\,.
\ea
$$
By Theorem \ref{T-NURough} (c), the set $F^{-1}_t(\{x\})$ is finite for a.e. $x\in\bR^N$, so that the formula
$$
\b(x):=\sum_{y\in F^{-1}_t(\{x\})}b(y)
$$
gives a measurable function $\b$ defined a.e. on $\bR^N$, and $\rho(t)=\b\scrL^N$. In particular 
$$
\rho(t)\ll\scrL^N\,,\quad\hbox{ and }\,\,\frac{d\rho(t)}{d\scrL^N}=\b\quad\hbox{ a.e. on }\bR^N\,.
$$
In view of the formula giving $b$, this is precisely statement (b). 

Let $\psi\in C_c(\bR^N_x\times\bR^N_\xi)$. Then
$$
\la\mu(t),\psi\ra=\la\mu^{in},\psi\circ\Phi_t\ra=\int_{\bR^N}\psi(F_t(y),\Xi_t(y,U^{in}(y)))\rho^{in}(y)dy\,.
$$
In terms of the function $b$ defined above, one has
$$
\la\mu(t),\psi\ra=\int_{\bR^N}\psi(F_t(y),\Xi_t(y,U^{in}(y)))b(y)J_t(y)dy
$$
and applying the area formula as above shows
$$
\ba
\la\mu(t),\psi\ra&=\int_{\bR^N}\left(\sum_{y\in F^{-1}_t(\{x\})}b(y)\psi(x,\Xi_t(y,U^{in}(y)))\right)dx
\\
&=\int_{\bR^N}\left(\sum_{y\in F^{-1}_t(\{x\})}b(y)\la\de_{\Xi_t(y,U^{in}(y))},\psi(x,\cdot)\ra\right)dx\,.
\ea
$$
This shows that $\mu(t)$ has a disintegration with respect to the Lebesgue measure $\scrL^N_x$ on $\bR^N_x$ and the canonical projection $\Pi$ given by
$$
\mu(t,x,\cdot):=\sum_{y\in F^{-1}_t(\{x\})}b(y)\de_{\Xi_t(y,U^{in}(y))}\quad\hbox{ for a.e. }x\in\bR^N\,.
$$
(Notice that the set $F^{-1}_t(\{x\})$ is finite for a.e. $x\in\bR^N$ by Theorem \ref{T-NURough} (c).) In view of the formula defining the measurable function
$b$, this last equality is precisely the conclusion in statement (c). The proof is complete.

\subsection{Proof of Theorem \ref{T-DecompRhot}}

Let $\psi\in C_c(\bR^N_x\times\bR^N_\xi)$ satisfy $\Supp(\psi)\cap\Lambda_t=\varnothing$. Then
$$
\la\mu(t),\psi\ra=\la\Phi_t\#\mu^{in},\psi\ra=\la\mu^{in},\psi\circ\Phi_t\ra=\int_{\bR^N}\psi(\Phi_t(y,U^{in}(y)))\rho^{in}(y)dy=0\,.
$$
Indeed
$$
\ba
\Supp(\psi)\cap\Lambda_t=\varnothing\Rightarrow\psi(\Lambda_t)=\{0\}&\Rightarrow\psi\circ\Phi_t(\Lambda^{in})=\{0\}
\\
&\Rightarrow\psi(\Phi_t(y,U^{in}(y)))=0\hbox{ for all }y\in\bR^N\,.
\ea
$$
Since $\la\mu(t),\psi\ra=0$ for each $\psi\in C_c(\bR^N_x\times\bR^N_\xi)$ such that $\Supp(\psi)\cap\Lambda_t=\varnothing$, one concludes that
$\Supp(\mu(t))\subset\Lambda_t$, which is statement (a).

Since $\scrL^N(E)=0$, one has $\rho^{in}=\rho^{in}\indc_{P_t}+\rho^{in}\indc_{Z_t}$ a.e. in $\bR^N$, and therefore
$$
\rho(t)=\rho_a(t)+\rho_s(t)\hbox{ with }\rho_a(t):=F_t\#(\rho^{in}\indc_{P_t}\scrL^N)\hbox{ and }\rho_s(t):=F_t\#(\rho^{in}\indc_{Z_t}\scrL^N)\,.
$$
By statement (b) in Theorem \ref{T-Mu(t)=}, one has $\rho_a(t)\ll\scrL^N$. 

Let us check that $\rho_s(t)$ is carried by $C_t$. Let $A\subset\bR^N$; then
$$
A\cap C_t=\varnothing\Rightarrow F_t^{-1}(A)\cap Z_t\subset F_t^{-1}(A)\cap(Z_t\cup E)=\varnothing\,,
$$
and
$$
\rho_s(t)(A)=\int_{F^{-1}_t(A)\cap Z_t}\rho^{in}(y)dy=0\,.
$$
Hence $\rho_s(t)$ is carried by $C_t$, and since $\scrL^N(C_t)=0$ by statement (e) in Theorem \ref{T-NURough}, we conclude that $\rho_s(t)\perp\scrL^N$. 
Thus, with $\rho_a(t)$ and $\rho_s(t)$ so defined, one has
$$
\rho_a(t)\ll\scrL^N\,,\quad\rho_s(t)\perp\scrL^N\,,\quad\hbox{ and }\rho(t)=\rho_a(t)=\rho_s(t)\,.
$$
Thus the pair $(\rho_a(t),\rho_s(t))$ is the Lebesgue decomposition of $\rho(t)$ with respect to $\scrL^N$, which is precisely statement (b). (For the uniqueness 
of the Lebesgue decomposition, see Theorem 6.10 (a) in \cite{Rudin}.) This completes the proof of Theorem \ref{T-DecompRhot}.

\subsection{Proof of Theorem \ref{T-Atom}}

If $x\in A_t$, then $\scrL^N(F_t^{-1}(\{x\})\cap Z_t)>0$; hence $\varnothing\not=F_T^{-1}(\{x\})\cap Z_t\subset F_T^{-1}(\{x\})\cap(Z_t\cup E)$ so that $x\in C_t$,
which proves statement (a).

Using (\ref{TranspRho}) and the Lebesgue decomposition in statement (b) of Theorem \ref{T-DecompRhot} shows that, for each $x\in\bR^N$
$$
\rho_a(t)(\{x\})=0\hbox{ and therefore }\rho(t)(\{x\})=\rho_s(t)(\{x\})=(\rho^{in}\indc_{Z_t}\scrL^N)(F_t^{-1}(\{x\}))\,.
$$
In other words
$$
\rho(t)(\{x\})=\int_{F_t^{-1}(\{x\})\cap Z_t}\rho^{in}(y)dy
$$
Hence
$$
\rho(t)(\{x\})>0\Rightarrow\scrL^{N}(F_t^{-1}(\{x\})\cap Z_t)>0\Rightarrow x\in A_t\,.
$$
Conversely, assume that $\rho^{in}>0$ a.e. on $Z_t$. The expression above for $\rho(t)(\{x\})$ shows that
$$
0=\rho(t)(\{x\})=\int_{Z_t}\rho^{in}(y)\indc_{F_t^{-1}(\{x\})}(y)dy\Rightarrow\indc_{F_t^{-1}(\{x\})}(y)=0\hbox{ for a.e. }y\in Z_t\,.
$$
In other words, $\scrL^N(F_t^{-1}(\{x\})\cap Z_t)=0$ and therefore $x\notin A_t$. This proves statement (b).

Let $\rho^{in}$ be a probability density such that $\rho^{in}>0$ a.e. on $Z_t$. By statement (b), 
$$
A_t=\{x\in\bR^N\,|\,\rho(t)(\{x\})>0\}=\bigcup_{n\ge 1}\left\{x\in\bR^N\,|\,\rho(t)(\{x\})\ge\frac1n\right\}\,.
$$
By the Bienaym\'e-Chebyshev theorem
$$
\#\left\{x\in\bR^N\,|\,\rho(t)(\{x\})\ge\frac1n\right\}\le n\rho(t)(\bR^N)=n
$$
so that $A_t$ is countable, being a denumerable union of finite sets. This proves statement (c) and concludes the proof of Theorem \ref{T-Atom}.

\section{Proof of Theorem \ref{T-WWKB}}\lb{S-Psi}

Denote $g_\eps(x):={(\pi\eps)^{-N/2}}e^{-|x|^2/\eps}$ and $G_\eps(x,\xi):=g_\eps(x)g_\eps(\xi)$. Then, one has
$$
G_\eps(x_0-x,\xi_0-\xi)=W_\eps[\Psi^{x_0,\xi_0}_\eps](x,\xi)\,,
$$
where
$$
\Psi^{x_0,\xi_0}_\eps(x):=(\pi\eps)^{-N/4}e^{-|x-x_0|^2/2\eps}e^{i\xi_0\cdot x/\eps}\,.
$$
Along with the Wigner transform $W_\eps[\psi_\eps(t,\cdot)]$, consider the Husimi transform
$$
\widetilde W_\eps[\psi_\eps(t,\cdot)]=W_\eps[\psi_\eps(t,\cdot)]\star_{x,\xi}G_\eps\,.
$$
A straightforward computation shows that, for each $x_0,\xi_0\in\bR^N$ and each $\eps>0$,
$$
\ba
\widetilde W_\eps[\psi_\eps(t,\cdot)](x_0,\xi_0)=\iint_{\bR^N\times\bR^N}W_\eps[\psi_\eps(t,\cdot)](x,\xi)\overline{W_\eps[\Psi^{x_0,\xi_0}_\eps](x,\xi)}dxd\xi&
\\
=\frac1{(2\pi\eps)^N}|\la\Psi^{x_0,\xi_0}_\eps|\psi_\eps(t,\cdot)\ra|^2\ge 0&\,,
\ea
$$
Therefore
\be\lb{MargiHusi}
\ba
\int_{\bR^N}\widetilde W_\eps[\psi_\eps(t,\cdot)](x_0,\xi_0)d\xi_0&=(g_\eps\star|\psi_\eps(t,\cdot)|^2)(x_0)\hbox{ and }
\\
\int_{\bR^N}\widetilde W_\eps[\psi_\eps(t,\cdot)](x_0,\xi_0)dx&=(g_\eps\star|\cF_\eps\psi_\eps(t,\cdot)|^2)(\xi_0)\,,
\ea
\ee
while
\be\lb{IntHusi}
\ba
\iint_{\bR^N\times\bR^N}\widetilde W_\eps[\psi_\eps(t,\cdot)](x_0,\xi_0)dx_0d\xi_0
=
\iint_{\bR^N\times\bR^N}|\la\Psi^{x_0,\xi_0}_\eps|\psi_\eps(t,\cdot)\ra|^2\frac{dx_0d\xi_0}{(2\pi\eps)^N}&
\\
=\int_{\bR^N}g_\eps\star|\psi_\eps(t,\cdot)|^2(x_0)dx_0
=
\|g_\eps\|_{L^1(\bR^N)}\|\psi_\eps(t,\cdot)\|^2_{L^2(\bR^N)}=1&\,,
\ea
\ee
where the penultimate equality follows from the conservation of the $L^2$ norm under the Schr\"odinger group.

On the other hand
\be\lb{Intmu}
\iint_{\bR^N\times\bR^N}\mu(t,dxd\xi)=\iint_{\bR^N\times\bR^N}\mu(0,dxd\xi)=\|a^{in}\|_{L^2(\bR^N)}=1
\ee
since $\mu(t)$ is the push-forward of the probability measure $\mu(0)$ under the Hamiltonian flow of $\tfrac12|\xi|^2+V(x)$.

Since $\widetilde W_\eps[\psi_\eps(t,\cdot)]\ge 0$ and
\be\lb{WConvHusi}
\iint_{\bR^N\times\bR^N}\widetilde W_\eps[\psi_\eps(t,\cdot)](x,\xi)\chi(x,\xi)dxd\xi\to\iint_{\bR^N\times\bR^N}\chi(x,\xi)\mu(t,dxd\xi)
\ee
for each $\chi\in C_c(\bR^N\times\bR^N)$ by Theorem III.1 (1) in \cite{LionsPaul}, we conclude from (\ref{IntHusi})-(\ref{Intmu}) that the convergence 
(\ref{WConvHusi}) holds for each $\chi\in C_b(\bR^N\times\bR^N)$ (see for instance Theorem 6.8 in chapter II of \cite{MalliavInt}).

On the other hand, for each $\chi\in C^1_b(\bR^N)$
$$
\ba
\left|\int_{\bR^N}\chi(x)(|\psi_\eps(t,x)|^2-|\psi_\eps(t,\cdot)|^2\star g_\eps(x))dx\right|&
\\
\le
\int_{\bR^N}|\chi(x)-\chi\star_xg_\eps(x)||\psi_\eps(t,x)|^2dx&
\\
\le
\sqrt{\eps}\|\grad\chi\|_{L^\infty}\int_{\bR^N}|y|g_1(y)dy\to 0&\,,
\ea
$$
and likewise
$$
\left|\int_{\bR^N}\chi(\xi)(|\cF_\eps\psi_\eps(t,\xi)|^2-|\cF_\eps\psi_\eps(t,\cdot)|^2\star g_\eps(\xi))d\xi\right|\to 0\,.
$$

We conclude from (\ref{MargiHusi}) and (\ref{WConvHusi}) that, for each $\chi\in C^1_b(\bR^N)$
\be\lb{LimIntWeps}
\ba
\int_{\bR^N}\chi(x)|\psi_\eps(t,x)|^2dx\to\iint_{\bR^N\times\bR^N}\chi(x)\mu(t,dxd\xi)
\\
\int_{\bR^N}\chi(\xi)|\cF_\eps\psi_\eps(t,\xi)|^2d\xi\to\iint_{\bR^N\times\bR^N}\chi(\xi)\mu(t,dxd\xi)
\ea
\ee
as $\eps\to 0$. On the other hand, 
$$
1=\int_{\bR^N}|\psi_\eps(t,x)|^2dx=\int_{\bR^N}|\cF_\eps\psi_\eps(t,\xi)|^2d\xi=\int_{\bR^N\times\bR^N}\mu(t,dxd\xi)
$$
so that (\ref{LimIntWeps}) holds for each $\chi\in C_b(\bR^N)$ by a standard density argument.

If $\chi(F_t(y))\rho^{in}(y)=0$ for a.e. $y\in Z_t$, one has
$$
\ba
\iint_{\bR^N\times\bR^N}\chi(x)\mu(t,dxd\xi)&=\int_{\bR^N}\chi(x)\rho(t,dx)=\int_{\bR^N}\chi(F_t(y))\rho^{in}(y)dy
\\
&=\int_{\bR^N}\chi(F_t(y))\rho^{in}(y)\indc_{P_t}(y)dy=\int_{\bR^N}\chi(x)\rho_a(t,dx)
\ea
$$
by Theorem \ref{T-DecompRhot}. With the first convergence statement in (\ref{LimIntWeps}) and Theorem \ref{T-Mu(t)=} (b), this equality implies statement (a).

If $\chi(\Xi_t(y,U^{in}(y)))\rho^{in}(y)=0$ for a.e. $y\in Z_t$, one has
$$
\ba
\iint_{\bR^N\times\bR^N}\chi(\xi)\mu(t,dxd\xi)&=\iint_{\bR^N\times\bR^N}\chi(\Xi_t(y,\eta))\mu^{in}(dyd\eta)
\\
&=\int_{\bR^N}\chi(\Xi_t(y,U^{in}(y)))\rho^{in}(y)dy
\\
&=\int_{\bR^N}\chi(\Xi_t(y,U^{in}(y)))\rho^{in}(y)\indc_{P_t}(y)dy\,.
\ea
$$
With the second convergence statement in (\ref{LimIntWeps}) and Theorem \ref{T-Mu(t)=} (c), this equality implies statement (b) and completes the proof
of Theorem \ref{T-WWKB}.

\section{Discussion of the examples}\lb{S-ProofEx}

In this section, we prove some of the statements in the examples presented in section \ref{S-Expl}. 

Example \ref{E-CantorFunc} is based on classical material on the Cantor function which can be found for instance in Exercise 1.6.47 of \cite{TaoMeas}.
Example \ref{E-Atom} is based on a straightforward computation and therefore needs no further discussion.

Examples \ref{E-Diffuse} and \ref{E-HausDim} are more involved and require detailed proofs.

\subsection{Proof of the statements in Example \ref{E-Diffuse}}

By regularity of the Lebesgue measure $\scrL^1$, there exists a compact set $K\subset(0,1)\setminus\bQ$ such that $\tfrac12<\scrL^1(K)\le 1$. Let 
$\Om=(0,1)\setminus K$; since $\Om$ is open in $(0,1)$ and contains $(0,1)\cap\bQ$, it is a countably infinite union of disjoint nonempty open intervals:
$$
\Om=\bigcup_{n\in\bN}I_n\,,\quad\hbox{Êso that }\scrL^1(\Om)=\sum_{n\in\bN}\scrL^1(I_n)\,.
$$
Besides $\scrL^1(I_n)>0$ for each $n\in\bN$ (indeed each $I_n$ is an open interval that contains at least one rational), so that $\l:=\scrL^1(\Om)>0$.
For each $n\ge 1$, we denote by $a_n$ and $b_n$ the endpoints of $I_n$, so that $I_n=(a_n,b_n)$.

Let $\chi\in C^\infty(\bR)$ satisfy the following properties:
$$
\Supp(\chi)\subset[-1,1]\,,\quad\hbox{ and }\chi(x)>0\hbox{ for all }x\in(-1,1)\,.
$$
Let $k\ge 1$; define
$$
g_n(z):=(b_n-a_n)^k\chi\left(\frac{2x-a_n-b_n}{b_n-a_n}\right)\,.
$$
Observe that, for each $j=1,\ldots,k-1$, one has
$$
\ba
\sum_{n\ge 0}\|g_n^{(j)}\|_{L^\infty(\bR)}&=\|\chi^{(j)}\|_{L^\infty(\bR)}\sum_{n\ge 0}2^j(b_n-a_n)^{k-j}
\\
&\le 2^j\|\chi^{(j)}\|_{L^\infty(\bR)}\sum_{n\ge 0}(b_n-a_n)
\\
&=2^j\|\chi^{(j)}\|_{L^\infty(\bR)}\scrL^1(\Om)\le 2^j\|\chi^{(j)}\|_{L^\infty(\bR)}\,,
\ea
$$
so that 
$$
g:=\sum_{n\ge 0}g_n\in C^{k-1}_b(\bR)\,.
$$
Pick $h\in C^\infty(\bR)$ s.t.
$$
h\rstr_{[0,1]}\equiv 0\,,\quad h\rstr_{(-\infty,-1]\cup[2,+\infty)}\equiv 1\hbox{ and }h>0\hbox{ on }\bR\setminus[0,1]\,,
$$
and set $f:=g+h$. Define
$$
U^{in}(y):=\int_0^yf(z)dz-y\quad\hbox{ for each }y\in\bR\,.
$$
By construction $U^{in}\in C^k(\bR)$, and one has
$$
U^{in}(y):=\left\{\ba{}&1-\int_{-1}^0h(z)dz&&\quad\hbox{ for }y<-1\,,\\ &\int_1^2h(z)dz+\int_0^1g(z)dz-2&&\quad\hbox{ for }y>2\,,\ea\right.
$$
so that $U^{in}\in C^k_b(\bR)$ and therefore satisfies the sublinearity condition (\ref{Sublin}).

With $H(x,\xi)=\tfrac12\xi^2$ so that $\Phi_t(x,\xi)=(x+t\xi,\xi)$, one has $F_t(y):=y+tU^{in}(y)$ so that, for $t=1$,
$$
F_1(y)=\int_0^yf(z)dz\quad\hbox{ for each }y\in\bR\,.
$$

\subsubsection{Proof of statement (a)}

Observe that $F_1\in C^k(\bR)$ and that $F_1'(y)=h(y)>0$ for each $y\in(-\infty,0)\cup(1,+\infty)$. On the other hand, for each $y_1,y_2\in[0,1]$ such that
$y_1<y_2$, the nonempty open interval $(y_1,y_2)\subset(0,1)$ contains at least one rational point. Therefore, $(y_1,y_2)\cap\Om\not=\varnothing$. In
particular, there exists a nonempty open interval $(\a,\b)\subset\Om\cap(y_1,y_2)$. Therefore $0<\a<\b<1$ and
$$
F_1(y_2)-F_1(y_1)\ge\int_\a^\b g(y)dy>0\,.
$$
Therefore $F_1$ is increasing on $\bR$. In particular, $F_1$ is one-to-one.

On the other hand $F_1(y)\sim y$ as $|y|\to\infty$ since $U^{in}$ is bounded on $\bR$. In particular
$$
F_1(y)\to-\infty\hbox{ as }y\to-\infty\,,\quad\hbox{ and }F_1(y)\to+\infty\hbox{ as }y\to+\infty\,.
$$
Since $F_1$ is continuous on $\bR$, it is onto by the intermediate values theorem.

\subsubsection{Proof of statement (b)}

By construction $F_1'(y)=h(y)>0$ if $y<0$ or if $y>1$. On the other hand, for each $y\in[0,1]$, one has $F_1'(y)=g(y)$, so that $F_1'(y)>0$ for all $y\in\Om$, 
while $F_1'(y)=0$ whenever $y\in K\cup\{0,1\}$. 

Since $F_1$ is one-to-one, the set of critical values of $F_1$, i.e. the caustic fiber $C_1$ at time $t=1$ is $C_1=F_1(K\cup\{0,1\})$.

\subsubsection{Conclusion}

Set $\rho^{in}=\tfrac1{1-\l}\indc_K$ and $\rho(t)=F_t\#(\rho^{in}\scrL^1)$ for each $t\in\bR$. In particular $\rho(1)=F_1\#\rho^{in}$ and 
$$
1=\rho(1)(\bR)\ge\rho(1)(C_1)=\|\rho^{in}\|_{L^1}=\frac1{1-\l}\scrL^1(K)=1\,.
$$
Since $\scrL^1(C_1)=0$ by Sard's theorem and $\rho(1)$ is a probability measure supported in $C_1$, we conclude that
$$
\rho(1)\perp\scrL^N\,.
$$ 

On the other hand, since $F_1$ is one-to-one and onto, for each $x\in\bR$, one has $\#F_1^{-1}(\{x\})=1$. In particular, for each $x\in C_1=F_1(K)$,
one has
$$
F_1^{-1}(\{x\})\subset(F'_1)^{-1}(\{0\})
$$
so that 
$$
\#(F_1^{-1}(\{x\})\cap (F'_1)^{-1}(\{0\}))=\#F_1^{-1}(\{x\})=1\,.
$$
In particular, for each $x\in C_1$, 
$$
(\rho^{in}\scrL^1)(F_1^{-1}(\{x\})\cap (F'_1)^{-1}(\{0\}))=0\,,
$$
which means that $x\notin A_1$. 

On the other hand, one has $\rho^{in}>0$ on $K$ and therefore a.e. on $Z_1=K\cup\{0,1\}$. By statement (b) in Theorem \ref{T-Atom}, one has
$$
\rho(1)(\{x\})=0\,.
$$
Hence $\rho(1)\perp\scrL^1$ and is diffuse, with $\rho(1)(C_1)=1$.

\subsection{Proof of the statements in Example \ref{E-HausDim}}

We begin with the following lemma, which is the key in understanding how the map $F_1$ acts on $[0,1]$. Obviously
$$
F_1(y)=y+U^{in}(y)=\left\{\ba{}&y&&\quad\hbox{ if }y\notin[0,1]\,,\\&\frac1\th\scrL^1(\Om(\th)\cap(0,y))&&\quad\hbox{ if }y\in(0,1)\,.\ea\right.
$$

\begin{Lem}\lb{L-F(a)} Consider the transformations 
$$
H:\,[0,1]\ni z\mapsto\th z\in[0,1]\,,\quad\hbox{ and }S:\,[0,1]\ni z\mapsto 1-z\in[0,1]\,.
$$
\smallskip
\noindent
a) One has $F_1\circ H=H\circ F_1$ and $F_1\circ S=S\circ F_1$ on $[0,1]$.

\noindent
b) The sequence of interval centers $a_{m,k}$ satisfies
$$
F_1(a_{m,k})=a_{m,k}\quad\hbox{ for all }m\ge 1\hbox{ and }k=1,\ldots,2^{m-1}\,.
$$
\end{Lem}

\begin{proof}[Proof of Lemma \ref{L-F(a)}]

Obviously $[0,1]\setminus E_1=(\th,1-\th)=:I_{1,1}$ so that 
$$
a_{1,1}=\tfrac12\,.
$$
The interval centers $a_{m,k}$ satisfy the following induction relations
$$
\left\{
\ba
{}&a_{m+1,k}=H(a_{m,k})\,,&&\quad 1\le k\le 2^{m-1}\,,
\\
&a_{m+1,k}=S(a_{m+1,2^m+1-k})\,,&&\quad 2^{m-1}+1\le k\le 2^m\,.
\ea
\right.
$$
Therefore 
$$
\left\{
\ba
{}&I_{m+1,k}=H(I_{m,k})\,,&&\quad 1\le k\le 2^{m-1}\,,
\\
&I_{m+1,k}=S(I_{m+1,2^m+1-k})\,,&&\quad 2^{m-1}+1\le k\le 2^m\,,
\ea
\right.
$$
and
$$
\left\{
\ba
{}&J_{m+1,k}=H(J_{m,k})\,,&&\quad 1\le k\le 2^{m-1}\,,
\\
&J_{m+1,k}=S(J_{m+1,2^m+1-k})\,,&&\quad 2^{m-1}+1\le k\le 2^m\,.
\ea
\right.
$$
In particular
$$
\Om(\th)=H(\Om(\th))\cup S\circ H(\Om(\th))\,,\quad H(\Om(\th))\cap S\circ H(\Om(\th))=\varnothing\,.
$$

Thus, for each $z\in[0,1]$, one has
$$
\Om(\th)\cap(0,H(z))\subset[0,\th]\
$$
so that 
$$
\Om(\th)\cap(0,H(z))=H(\Om(\th))\cap(0,H(z))=H(\Om(\th)\cap(0,z))\,,
$$
and therefore
$$
\ba
F_1(H(z))=\frac1\th\scrL^1(\Om(\th)\cap(0,H(z)))&=\frac1\th\scrL^1(H(\Om(\th)\cap(0,z)))
\\
&=\scrL^1(\Om(\th)\cap(0,z))=\th F_1(z)=H(F_1(z))\,.
\ea
$$
Likewise, for each $z\in[0,1]$,
$$
\Om(\th)\cap(S(z),1)=\Om(\th)\cap S(0,z)=S(\Om(\th))\cap S((0,z))=S(\Om(\th)\cap(0,z))\,,
$$
so that
$$
\ba
F_1(S(z))&=\frac1\th\scrL^1(\Om(\th)\cap(0,S(z)))=\frac1\th\scrL^1(\Om(\th))-\frac1\th\scrL^1(\Om(\th)\cap[S(z),1))
\\
&=1-\frac1\th\scrL^1(\Om(\th)\cap(S(z),1))=1-\frac1\th\scrL^1(S(\Om(\th)\cap(0,z)))
\\
&=1-F_1(z)=S(F_1(z))\,.
\ea
$$
This proves statement (a).

As for statement (b), we first observe that $a_{1,1}=\tfrac12=S(\tfrac12)=S(a_{1,1})$. Therefore
$$
F(a_{1,1})=F(S(a_{1,1}))=1-F(a_{1,1})\hbox{ so that }F(a_{1,1})=\tfrac12=a_{1,1}\,.
$$
This observation, together with the induction relations on the sequence of $a_{m,k}$ and the commutation properties of $F_1$ with $H$ and $S$, implies
that each interval center $a_{m,k}$ is a fixed point of $F_1$.
\end{proof}

\subsection{Proof of statement (a) in Example \ref{E-HausDim}}

The vector field $U^{in}$ is Lipschitz continuous on $[0,1]$ (as the antiderivative of the bounded measurable function $\frac1\th\indc_\Om-1$). Its extension
by $0$ to the $\bR\setminus[0,1]$ is Lipschitz continuous on $\bR$ since $\scrL^1(\Om)=\th$.

The function $F_1:\bR\ni y\mapsto y+U^{in}(y)\in\bR$ is therefore Lipschitz continuous. It is nondecreasing on $\bR$ (as the antiderivative of the nonnegative
measurable function $\indc_{(-\infty,0)\cup(1,+\infty)}+\frac1\th\indc_{\Om(\th)}$). Since $\Om(\th)$ is an open subset of $(0,1)$, the map $F_1$ is differentiable 
on $(-\infty,0)\cup\Om(\th)\cup(1,+\infty)$ and one has
$$
F_1'(y)=\left\{\ba{}&1&&\quad\hbox{ if }y\in(-\infty,0)\cup(1,+\infty)\,,\\ &1/\th&&\quad\hbox{ if }y\in\Om(\th)\,.\ea\right.
$$
In particular, $F_1'>0$ on $(-\infty,0)\cup\Om(\th)\cup(1,+\infty)$. Since we already know that $F_1$ is nondecreasing on $\bR$, we infer that the map $F_1$ 
is increasing on $(-\infty,0)\cup\Om(\th)\cup(1,+\infty)$. This proves statement (a).

\subsection{Proof of statement (b) in Example \ref{E-HausDim}}

Observe that $F_1(a_{m,k})=a_{m,k}$ for all $m\ge 1$ and all $k=1,\ldots,2^{m-1}$, while $F_1'(y)=\frac1\th$ for all $y\in\Om(\th)$. Therefore 
$F_1(J_{m,k})=I_{m,k}$, and
$$
F_1(\Om(\th))=\bigcup_{1\le k\le 2^{m-1}\atop m\ge 1}I_{m,k}=[0,1]\setminus K(\th)\,.
$$
Since $F_1$ coincides with the identity map on $(-\infty,0)$ and on $(1,+\infty)$, we conclude
$$
F_1((-\infty,0)\cup\Om(\th)\cup(1,+\infty))=\bR\setminus K(\th)\,.
$$
Since $F_1$ is continuous and $F_1(y)\to\pm\infty$ as $y\to\pm\infty$, we conclude that $F_1$ is onto by the intermediate values theorem. Therefore
$$
F_1(\tilde K(\th))=K(\th)
$$
which concludes the proof of statement (b).

\subsection{Proof of statement (c) in Example \ref{E-HausDim}}
Set
$$
\cO(\th):=[0,1]\setminus K(\th)=\bigcup_{1\le k\le 2^{m-1}\atop m\ge 1}I_{m,k}
$$
so that
$$
\ba
\bigcup_{1\le k\le 2^{m-1}\atop m\ge 1}((a_{m,k}-r_m,a_{m,k}-\th r_m)\cup(a_{m,k}+\th r_m,a_{m,k}+r_m))&
\\
\subset\cO(\th)\setminus\Om(\th)\subset[0,1]\setminus\Om(\th)=\tilde K(\th)&\,.
\ea
$$
By Lemma \ref{L-F(a)}, one has $F_1(a_{m,k})=a_{m,k}$ for each $m\ge 1$ and each $k=1,\ldots,2^{m-1}$; besides, 
$$
y\in J_{m,k}\subset\Om(\th)\Rightarrow F'_1(y)=\frac1\th\,,\quad
\hbox{ while }
y\in I_{m,k}\setminus\overline{J_{m,k}}\subset\tilde K(\th)\Rightarrow F'_1(y)=0\,.
$$
Hence
$$
\left\{
\ba
{}&F'_1(y)=F'_1(a_{m,k}-\th r_m)=a_{m,k}-r_m\hbox{ for all }y\in(a_{m,k}-r_m,a_{m,k}-\th r_m)\,,
\\
&F'_1(y)=F'_1(a_{m,k}+\th r_m)=a_{m,k}+r_m\hbox{ for all }y\in(a_{m,k}+\th r_m,a_{m,k}+r_m)\,.
\ea
\right.
$$
Thus, for each $\phi\in C_c(\bR^N\times\bR^N)$
$$
\ba
\iint_{\bR^N\times\bR^N}\phi(x,\xi)\mu(1,dxd\xi)&=\iint_{\bR^N\times\bR^N}\phi(y+\xi,\xi)\mu^{in}(dxd\xi)
\\
&=\frac1{1-\th}\int_{\tilde K(\th)}\phi(F_1(y),U^{in}(y))dy
\\
&=\frac1{1-\th}\int_{\tilde K(\th)}\phi(F_1(y),F_1(y)-y)dy\,.
\ea
$$
If $\phi(x,\xi)\ge 0$ for all $x,\xi\in\bR^N$, one has
$$
\ba
\iint_{\bR^N\times\bR^N}\phi(x,\xi)\mu(1,dxd\xi)
\ge\tfrac1{1-\th}\sum_{m\ge 1}\sum_{k=1}^{2^{m-1}}\int_{a_{m,k}-r_m}^{a_{m,k}-\th r_m}\phi(F_1(y),F_1(y)-y)dy&
\\
+
\tfrac1{1-\th}\sum_{m\ge 1}\sum_{k=1}^{2^{m-1}}\int_{a_{m,k}+\th r_m}^{a_{m,k}+r_m}\phi(F_1(y),F_1(y)-y)dy&
\\
\ge\tfrac1{1-\th}\sum_{m\ge 1}\sum_{k=1}^{2^{m-1}}\int_{a_{m,k}-r_m}^{a_{m,k}-\th r_m}\phi(a_{m,k}-r_m,a_{m,k}-r_m-y)dy&
\\
+\tfrac1{1-\th}\sum_{m\ge 1}\sum_{k=1}^{2^{m-1}}\int_{a_{m,k}+\th r_m}^{a_{m,k}+r_m}\phi(a_{m,k}+r_m,a_{m,k}+r_m-y)dy&
\\
=\tfrac1{1-\th}\sum_{m\ge 1}\sum_{k=1}^{2^{m-1}}\left(\int_{-(1-\th)r_m}^{0}\phi(a_{m,k}-r_m,z)dz+ \int_{0}^{(1-\th)r_m}\phi(a_{m,k}+r_m,z)dz\right)&\,.
\ea
$$
Thus
$$
\mu(1)\ge\frac1{1-\th}\sum_{m\ge 1}\sum_{k=1}^{2^{m-1}}(\de_{a_{m,k}-r_m}\otimes\indc_{(-(1-\th)r_m,0)}+\de_{a_{m,k}+r_m}\otimes\indc_{(0,(1-\th)r_m)})\,.
$$

On the other hand
$$
\ba
\La\frac1{1-\th}\sum_{m\ge 1}\sum_{k=1}^{2^{m-1}}(\de_{a_{m,k}-r_m}\otimes\indc_{(-(1-\th)r_m,0)}+\de_{a_{m,k}+r_m}\otimes\indc_{(0,(1-\th)r_m)}),1\Ra&
\\
=\frac1{1-\th}\sum_{m\ge 1}\sum_{k=1}^{2^{m-1}}2(1-\th)r_m=\sum_{m\ge 1}\sum_{k=1}^{2^{m-1}}2r_m
	=\sum_{m\ge 1}\sum_{k=1}^{2^{m-1}}2\tfrac12(1-2\th)\th^{m-1}&
\\
=(1-2\th)\sum_{m\ge 1}\sum_{k=1}^{2^{m-1}}\th^{m-1}=(1-2\th)\sum_{m\ge 1}2^{m-1}\th^{m-1}=1&\,.
\ea
$$
Since
$$
\iint_{\bR^N\times\bR^N}\mu(1,dxd\xi)=1\,,
$$
the inequality above is in fact an equality, which is precisely statement (c).

\subsection{Proof of statement (d) in Example \ref{E-HausDim}}

Observe that
$$
\rho(1)=\Pi\#\mu(1)=\sum_{m\ge 1}r_m\sum_{k=1}^{2^{m-1}}(\de_{a_{m,k}-r_m}+\de_{a_{m,k}+r_m}\,,
$$
which gives the first equality in statement (d), since $r_m=\tfrac12(1-2\th)\th^{m-1}$. Thus 
$$
\rho(1)(\bR^N)=\mu(1)(\bR^N\times\bR^N)=\mu^{in}(\bR^N\times\bR^N)=\int_{\bR^N}\rho^{in}(y)dy=1
$$
and 
$$
\Supp(\rho(1))=\overline{\{a_{m,k}\pm r_m\,,|\,m\ge 1\,,\,\, k=1,\ldots,2^{m-1}\}}\subset K(\th)\,.
$$
Since $\scrL^1(K(\th))=0$, this implies that $\rho(1)\bot\scrL^1$.

It remains to prove that
$$
K(\th)\subset\overline{\{a_{m,k}\pm r_m\,,|\,m\ge 1\,,\,\, k=1,\ldots,2^{m-1}\}}\,.
$$
Let $z\in K(\th)$ and let $\eps>0$; pick $n\ge 1$ large enough so that $\th^n<\eps$. We recall that $K(\th)\subset E_n$ and that $E_n$ is the union of $2^n$ 
closed segments of length $\th^n$. The set of endpoints of these segments is $P_n\cup\{0,1\}$, where
$$
P_n:=\{a_{m,k}\pm r_m\,|\,m=1,\ldots,n\hbox{ and }k=1,\ldots,2^{m-1}\}\,.
$$
Hence 
$$
\Dist(z,P_n)\le\th^n<\eps\,,
$$
which shows that $P_n$ is dense in $K(\th)$, and concludes the proof.

\subsection{Proof of statement (e) in Example \ref{E-HausDim}}

For each $n\ge 1$, set
$$
E_n:=[0,1]\setminus\bigcup_{1\le k\le 2^{m-1}\atop 1\le m\le n}I_{m,k}\,.
$$
Then 
$$
E_n=\bigcup_{1\le k\le 2^n}[\a_{n,k},\b_{n,k}]
$$
with 
$$
0\le\a_{n,1}<\b_{n,1}<\ldots<\a_{n,2^n}<\b_{n,2^n}=1\,,\quad\b_{n,k}-\a_{n,k}=\th^n\,.
$$
Obviously
$$
\{\a_{n,k}|\,1<k\le 2^n\}\cup\{\b_{n,k}|\,1\le k<2^n\}=\{a_{m,k}\pm r_m\,|\,1\le k\le 2^{m-1}\,,\,\,1\le m\le n\}\,.
$$
By construction $F_1$ is differentiable on $I_{m,k}\setminus\overline{J_{m,k}}$, and one has
$$
F'_1=0\hbox{ on }I_{m,k}\setminus\overline{J_{m,k}}=(a_{m,k}-r_m,a_{m,k}-\th r_m)\cup(a_{m,k}+\th r_m,a_{m,k}+r_m)\,.
$$
An elementary computation shows that $r_m=\rho\th^{m-1}$, so that
$$
F'_1=0\hbox{ on }(\a_{n,k}-\rho\th^{n-1},\a_{n,k})\cup(\b_{n,k},\b_{n,k}+\rho\th^{n-1})\,.
$$

Let $x\in K(\th)\setminus\{0,1\}$. First, if $x\in\{\a_{n,k}|\,1<k\le 2^n\}\cup\{\b_{n,k}|\,1\le k<2^n\}$, then
$$
\ba
\{x\}=F_1((\a_{n,k}-\rho\th^{n-1},\a_{n,k}))\hbox{ if }x=\a_{n,k}\,,
\\
\{x\}=F_1((\b_{n,k},\b_{n,k}+\rho\th^{n-1}))\hbox{ if }x=\b_{n,k}\,,
\ea
$$
so that $x\in F_1(Z_1)$, i.e. $x$ is a critical value of $F_1$.

Next, assume that
$$
x\in K(\th)\setminus\{\a_{n,k},\b_{n,k}\,|\,1\le k\le 2^n\}\,.
$$
Since 
$$
\{\a_{n,k},\b_{n,k}\,|\,1\le k\le 2^n\,,\,\,n\ge 1\}\hbox{ is dense in }K(\th)
$$
and $\b_{n,k}-\a_{n,k}=\th^n\to 0$ as $n\to\infty$, there exists a sequence $(k_n)_{n\ge 1}$ such that $x\in(\a_{n,k_n},\b_{n,k_n})$.

Since $F_1$ is continuous  and $F_1(\a_{n,k_n})=\a_{n,k_n}$ while $F_1(\b_{n,k_n})=\b_{n,k_n}$, there exists $y\in(\a_{n,k_n},\b_{n,k_n})$ such that
$F_1(y)=x$.

Assume that $F_1$ is differentiable at $y$. For each $n\ge 1$, pick $\xi_n,\eta_n$ so that
$$
\xi_n,\eta_n\in(\a_{n,k_n}-\rho\th^{n-1},\a_{n,k_n})\hbox{ and }\eta_n-\xi_n=\rho\th^n\,.
$$
Then
$$
\ba
0=\frac{F_1(\eta_n)-F_1(\xi_n)}{\eta_n-\xi_n}&=\frac{F_1(\eta_n)-F_1(y)}{\eta_n-y}\frac{\eta_n-y}{\eta_n-\xi_n}
+\frac{F_1(y)-F_1(\xi_n)}{y-\xi_n}\frac{y-\xi_n}{\eta_n-\xi_n}
\\
&=(F'_1(y)+\om(\eta_n-y))\frac{\eta_n-y}{\eta_n-\xi_n}+(F'_1(y)+\om(y-\xi_n))\frac{y-\xi_n}{\eta_n-\xi_n}
\\
&=F'_1(y)+\om(\eta_n-y)\frac{\eta_n-y}{\eta_n-\xi_n}+\om(y-\xi_n)\frac{y-\xi_n}{\eta_n-\xi_n}
\ea
$$
where $\om(r)\to 0$ as $r\to 0$. Observe that
$$
\ba
\left|\frac{\eta_n-y}{\eta_n-\xi_n}\right|\le\frac{|\eta_n-y|}{\rho\th^n}\le\frac{\rho\th^{n-1}+\th^n}{\rho\th^n}\le\frac1\th+\frac1\rho\,,
\\
\left|\frac{y-\xi_n}{\eta_n-\xi_n}\right|\le\frac{|y-\xi_n|}{\rho\th^n}\le\frac{\rho\th^{n-1}+\th^n}{\rho\th^n}\le\frac1\th+\frac1\rho\,,
\ea
$$
so that
$$
\om(\eta_n-y)\frac{\eta_n-y}{\eta_n-\xi_n}+\om(y-\xi_n)\frac{y-\xi_n}{\eta_n-\xi_n}\to 0
$$
as $n\to\infty$. 

Thus, if $F_1$ is differentiable at $y$, then $F'_1(y)=0$, so that $x=F_1(y)\in F_1(Z_1)$. Otherwise $y\in E$, so that $x=F_1(y)\in F_1(E)$. 

In all cases, one has $x\in C_1$, so that $K(\th)\setminus\{0,1\}\subset C_1$. Since $F_1$ is differentiable and $F'_1>0$ on $(-\infty,0)\cup\Om(\th)\cup(1,\infty)$ 
while $F_1((-\infty,0)\cup\Om(\th)\cup(1,\infty))=\bR\setminus K(\th)$, one has also $C_1\subset K(\th)$.


\bigskip
\noindent
\textbf{Acknowledgement.} Peter  Markowich thanks the Fondation des Sciences Math\'e\-matiques de Paris for its support during the preparation of this paper.


\end{document}